\documentclass[10pt,reqno]{amsart}

\usepackage{amsmath,amssymb,amsfonts,latexsym,amsthm}

\usepackage{mathrsfs}

\numberwithin{equation}{section}

\usepackage{graphicx}

\theoremstyle{plain}

\newtheorem{lemma}{Lemma}[section]
\newtheorem{theorem}[lemma]{Theorem}
\newtheorem{proposition}[lemma]{Proposition}
\newtheorem{corollary}[lemma]{Corollary}

\theoremstyle{definition}

\newtheorem{remark}[lemma]{Remark}
\newtheorem{definition}[lemma]{Definition}

\theoremstyle{remark}

\newcommand{\A}{\mathbb{A}}
\newcommand{\R}{\mathbb{R}}
\newcommand{\C}{\mathbb{C}}
\newcommand{\Z}{\mathbb{Z}}
\newcommand{\N}{\mathbb{N}}

\newcommand{\cL}{{\mathcal{L}}}
\newcommand{\cD}{{\mathcal{D}}}
\newcommand{\cR}{{\mathcal{R}}}

\newcommand{\ccC}{\mathscr{C}}
\newcommand{\ccB}{\mathscr{B}}

\newcommand{\tTh}{\widetilde{\Theta}}

\renewcommand{\Re}{\mathrm{Re}\,}

\newcommand{\ind}{\mathrm{ind}\,}

\newcommand{\nul}{\mathrm{nul}\,}

\newcommand{\ess}{\sigma_\mathrm{\tiny{ess}}}
\newcommand{\ptsp}{\sigma_\mathrm{\tiny{pt}}}

\newcommand{\sppi}{\sigma_\pi}
\newcommand{\spd}{\sigma_\delta}

\newcommand{\<}{\langle}
\renewcommand{\>}{\rangle}

\begin{document}

\title[Stability of diffusion-degenerate Fisher-KPP traveling fronts]{Spectral stability of traveling 
fronts for reaction diffusion-degenerate Fisher-KPP equations}

\author[J.F. Leyva]{J. Francisco Leyva}
 
\address{{\rm (J. F. Leyva)} Posgrado en Ciencias Matem\'{a}ticas\\Universidad Nacional Aut\'{o}noma 
de M\'{e}xico\\Circuito Exterior s/n, Ciudad Universitaria, C.P. 04510 Cd. de M\'{e}xico (Mexico)}

\email{jfleyva.84@gmail.com}

\author[R.G. Plaza]{Ram\'on G. Plaza}

\address{{\rm (R. G. Plaza)} Instituto de 
Investigaciones en Matem\'aticas Aplicadas y en Sistemas\\Universidad Nacional Aut\'onoma de 
M\'exico\\Circuito Escolar s/n, Ciudad Universitaria, C.P. 04510 Cd. de M\'{e}xico (Mexico)}

\email{plaza@mym.iimas.unam.mx}

\begin{abstract}
This paper establishes the spectral stability in exponentially weighted spaces of smooth traveling monotone fronts for reaction diffusion equations of Fisher-KPP type with nonlinear degenerate diffusion coefficient. It is assumed that the former is degenerate, that is, it vanishes at zero, which is one of the equilibrium points of the reaction. A parabolic regularization technique is introduced in order to locate a subset of the compression spectrum of the linearized operator around the wave, whereas the point spectrum is proved to be stable with the use of energy estimates. Detailed asymptotic decay estimates of solutions to spectral equations are required in order to close the energy estimates. It is shown that all fronts traveling with speed above a threshold value are spectrally stable in an appropriately chosen exponentially weighted $L^2$-space.
\end{abstract}



\maketitle

\setcounter{tocdepth}{1}



\section{Introduction}

In this paper we study scalar reaction-diffusion equations of the form
\begin{equation}
\label{degRD}
 u_t = (D(u)u_x)_x + f(u),
\end{equation}
where $u = u(x,t) \in \R$, $x \in \R$, $t > 0$, and the diffusion coefficient $D = D(u)$ is a nonlinear, 
non-negative density dependent function which is \textit{degenerate} at $u = 0$. More precisely, it is assumed that 
$D$ 
satisfies 
\begin{equation}
\label{hypD}
\begin{aligned}
& D(0) = 0, \;\; \, D(u) > 0 \; \; \text{for all} \, u \in (0,1],\\
& D \in C^2([0,1];\R) \;\; \text{with} \; D'(u) > 0 \; \text{for all} \; u \in [0,1].
\end{aligned}
\end{equation}
As an example we have the quadratic function 
\begin{equation}
\label{Dbeta}
 D(u) = u^2 + b u,
\end{equation}
for some constant $b > 0$, as proposed by Shigesada \textit{et al.} \cite{SKT79} to model dispersive forces 
due to mutual interferences between individuals of an animal population.

The nonlinear reaction function is supposed to be of \textit{Fisher-KPP type} \cite{Fis37,KPP37}, that is, $f 
\in C^2([0,1];\R)$ has one stable and one unstable equilibrium points in $[0,1]$; more precisely,
\begin{equation}
 \label{Fisherf}
 \begin{aligned}
	&f(0)=f(1)=0,\\
	&f'(0) > 0, \;\, f'(1)<0,\\
	&f(u)>0\textrm{ for all } u \in (0,1).
	\end{aligned}
\end{equation}
An example of a reaction of Fisher-KPP type is the classical logistic function,
\begin{equation}
 \label{logistic}
f(u) = u (1-u),
\end{equation}
which is used to model the dynamics of a population in an environment with limited resources.


Nonlinear reaction-diffusion equations arise as models in several natural phenomena, with applications to 
population dynamics, chemical reactions with diffusion, fluid mechanics, action potential propagation, and 
flow in porous media (see \cite{Fowl97,MurI3ed,Muskat37} and the references therein). The celebrated 
equation $u_t = u_{xx} + u(1-u)$, known as the \textit{Fisher-KPP reaction diffusion equation}, was 
introduced in seminal works by Fisher \cite{Fis37} and by Kolmogorov, Petrosky and Piskunov 
\cite{KPP37} as a model to describe the spatial one-dimensional spreading when mutant individuals with higher 
adaptability appear in populations. Since then, very intensive research has been carried out to extend their 
model and to take into account other physical, chemical or biological factors. The first reaction-diffusion 
models of the form \eqref{degRD} considered constant diffusion coefficients $D \equiv D_0 > 0$ (cf. 
\cite{MurI3ed,Ske51}). It is now clear, however, that there are situations in which the diffusion coefficient 
must be a function of the unknown $u$. For example, in the context of population biology it has been reported 
(cf. \cite{Aron80,GuMaC77,MyKr74}) that motility varies with the population density, requiring 
density-dependent dispersal coefficients. Such feature has been incorporated into mathematical models in 
spatial ecology \cite{SKT79}, and eukaryotic cell biology \cite{SePO07}, to mention a few. Density dependent 
diffusion functions seem to be particularly important in bacterial aggregation, where motility appears as an 
increasing function of bacterial density as well as of nutrients or other substrate substances (cf. 
\cite{B-JCoLev00,KMMUS,LMP1}). A significant example from a different field is the equation $u_t = \Delta 
(u^m)$, with $m > 0$, which is very well-known in chemical engineering for the description of porous media 
\cite{Muskat37}. 

An interesting phenomenon occurs when the nonlinearity in the diffusion coefficient is 
\textit{degenerate}, meaning that diffusion approaches zero when the density does also. Among the new 
mathematical features one finds that equations with degenerate diffusion possess finite speed of propagation 
of initial disturbances, in contrast with the strictly parabolic case. Another property is the emergence of 
traveling waves of ``sharp'' type (cf. \cite{SaMaKa96b,Sh10}). Reaction-diffusion 
models with degenerate nonlinear diffusion are widely used nowadays to describe biological phenomena 
(see, for example, \cite{GuMaC77,MurI3ed} and the references therein). 

In all these models, one of the most important mathematical solution types is the traveling front. Traveling fronts (or wave fronts) are solutions to equations \eqref{degRD}, of the form
\[
u(x,t) = \varphi(x - ct),
\]
where $c \in \R$ is the speed of the wave and $\varphi : \R \to \R$ is the wave profile function. For pattern 
formation problems it is natural to consider infinite domains and to neglect the influence of boundary 
conditions. Thus, these fronts usually have asymptotic limits, $u_\pm = \lim_{\xi \to \pm \infty} 
\varphi(\xi)$, which are equilibrium points of the reaction function under consideration, $f(u_\pm) = 0$. They 
are widely used to model, for example, invasions in theoretical ecology \cite{Hast-etal-05}, the advancing 
edges of cell populations like growing tumors \cite{ShCh}, or the envelope fronts of certain bacterial 
colonies which extend effectively as one-dimensional fronts \cite{KMMUS,LMP1}. Since the classical work of 
Kolmogorov \textit{et al.} \cite{KPP37}, which set the foundations of their existence theory, traveling waves 
solutions to equations of the form \eqref{degRD} have attracted a great deal of attention. The existence of 
fronts for reaction-diffusion equations with degenerate diffusion was first studied in particular cases (see, 
e.g., \cite{Aron85,New80,New83}). The first general existence results for degenerate diffusions satisfying 
hypotheses \eqref{hypD}, and generic reaction functions of Fisher-KPP type satisfying \eqref{Fisherf}, are due 
to Sanchez-Gardu\~{n}o and Maini \cite{SaMa94a,SaMa95}. In these works, the authors prove the existence of a 
positive threshold speed $c_* > 0$ such that: (i) there exist no traveling fronts with speed $0 < c < c_*$; 
(ii) there exists a traveling wave of \textit{sharp} type traveling with speed $c = c_*$, with 
$\varphi(-\infty) = 1$ and $\varphi(\xi) = 0$ for all $\xi \geq \xi_*$ with some $\xi_* \in \R$; and, (iii) 
there exists a family of smooth monotone decreasing traveling fronts, each of which travels with speed $c > 
c_*$, and is such that $\varphi(-\infty) = 1$ and $\varphi(+\infty) = 0$ (see Proposition \ref{propSaMa} 
below). The present paper pertains to the stability properties of monotone Fisher-KPP degenerate fronts (case 
(iii) above).


The stability of traveling wave solutions is a fundamental issue. It has been addressed for strictly parabolic reaction-diffusion equations using methods that range from comparison principles for super- and sub-solutions (see, e.g., the pioneer work of Fife and McLeod \cite{FiM77}), linearization techniques and generation of stable semigroups \cite{Sat}, and dynamical systems techniques for PDEs \cite{He81}, among others. The modern stability theory of nonlinear waves links the functional analysis approach with dynamical systems techniques, by setting a program leading to the spectral stability properties (the analysis of the spectrum of the linearized differential operator around the wave) and their relation to the nonlinear (orbital) stability of the waves under the dynamical viewpoint of the equations of evolution. The reader is referred to the seminal paper by Alexander, Gardner and Jones \cite{AGJ90}, the review article by Sandstede \cite{San02}, and the recent book by Kapitula and Promislow \cite{KaPro13} for further information. A recent contribution by Meyries \textit{et al.} \cite{MRS14} follows the same methodology, and presents rigorous results for quasi-linear systems with density-dependent diffusion tensors which are strictly parabolic (non-degenerate). 

In this paper we take a further step in this general stability program by considering degenerate diffusion 
coefficients. We start with the study of scalar equations and specialize the analysis to the spectral 
stability of the fronts. The former property, formally defined as the absence of spectra with positive real 
part of the linearized differential operator around the wave (see Definition \ref{defspecstab} below), can be 
seen as a first step of the general program. The degeneracy of the diffusion coefficient in one of 
the asymptotic limits poses some technical difficulties which are not present in the standard parabolic case. 
As far as we know, this is the first time that the spectral stability of a degenerate front is addressed in 
the literature. The contributions of this paper can be summarized as follows:
\begin{itemize}
\item[-] Due to the degeneracy of the diffusion at one of the equilibrium points of the reaction, the hyperbolicity 
of the asymptotic coefficients at one of the end points, which arise when the spectral problem is written in 
first order form, is lost. 
This precludes a direct application of the standard methods to locate the essential spectrum of the 
linearized 
operator around the front (cf. \cite{KaPro13,MRS14,San02}). To circumvent this difficulty, we propose an 
equivalent 
(but \textit{ad hoc}) partition of the spectrum of the linearized operator in the form $\sigma = \ptsp 
\cup \sigma_\delta \cup \sigma_\pi$, where $\ptsp$ is the point spectrum, $\sigma_\pi$ is a subset of the approximate spectrum, and 
$\sigma_\delta$ is a subset of the compression spectrum (see Definition \ref{defspec} below for
details).
\item[-] With the use of energy estimates of solutions to spectral equations, we control the point spectrum. For that purpose, we introduce a suitable 
transformation that gets rid of some of the advection terms, allowing us to close the energy estimate and to locate the eigenvalues along the non-positive real line. A detailed analysis (included in the Appendix \ref{secsoldecay}) of the 
decay properties of the associated eigenfunctions is necessary to justify such transformation. 
\item[-]  In order to locate the subset $\sigma_\delta$ of the compression spectrum, we introduce a regularization 
technique which circumvents the degeneracy of the diffusion at one of the equilibrium points. It is shown that the family of regularized operators converge, in the generalized sense, to the degenerate operator as the regularization parameter tends to zero. This convergence allows us, in 
turn, to relate the standard Fredholm properties of the regularized operators to those of the original degenerate operator. 
\item[-] As it is known from the parabolic Fisher-KPP case \cite{He81,KaPro13}, the spectral stability of fronts holds 
in exponentially weighted spaces only. We thus show that one can choose an appropriate weighted $L^2$-space 
(provided that a certain condition on the speed holds) in which degenerate-diffusion fronts are spectrally 
stable. To that end, we profit from the invariance of the point spectrum under conjugation, from the 
particular technique to locate $\sigma_\delta$ based on its Fredholm borders, and from the particular choice of the exponential weight, which allows us to control the subset $\sigma_\pi$ of the approximate spectrum. 
\end{itemize}

\smallskip

The main result of this paper is the content of the following
\begin{theorem}
\label{mainthm}
For any monotone traveling front for Fisher-KPP reaction diffusion-degenerate equations \eqref{degRD}, under hypotheses \eqref{hypD} and \eqref{Fisherf}, and
traveling with speed $c \in \R$ satisfying the condition
\[
c > \max \Big\{ \, c_*, \, \,\frac{f'(0) \sqrt{D(1)}}{\sqrt{f'(0) - f'(1)}} \,\Big\} > 0, 
\]
there exists an exponentially weighted space $L_a^2(\R;\C) = \{ u \, : \, e^{ax}u(x) \in L^2(\R;\C)\}$, with $a \in \R$, such that the front is 
$L_a^2$-spectrally stable. Here $c_* > 0$ denotes the minimum threshold speed (the velocity of the sharp wave).
\end{theorem}

\begin{remark}
In the most general case, the existence theory of \cite{SaMa94a} does not provide a precise analytic expression for $c_* > 0$, the threshold speed of sharp waves. The explicit form of $c_*$, however, is known in some particular cases, for specific choices of the functions $D$ and $f$ (see, e.g., \cite{Aron80,GiKe04,SaMa94b}).
\end{remark}

There exist previous results on the stability of diffusion-degenerate fronts in the literature. In an early 
paper, Hosono \cite{Hos86} addresses the convergence to traveling fronts for reaction 
diffusion equations in the ``porous medium'' form, $u_t = (u^m)_{xx} + f(u)$, with $m > 0$ (that is, for 
$D(u) = mu^{m-1}$) and reaction function $f$ of Nagumo (or bistable type). His method is 
based on the construction of super- and sub-solutions to the parabolic problem and the use of the comparison 
principle. Hosono establishes the asymptotic convergence of solutions to the nonlinear equation to a 
translated front when the initial data is close to the stationary front profile. We observe that this particular  
diffusion coefficient does not satisfy assumptions \eqref{hypD}; in addition, the method of proof relies 
heavily on the particular properties of solutions to the porous medium equation. Hosono's paper, however, 
warrants note as the first work containing a rigorous proof of convergence to a traveling front for reaction 
diffusion-degenerate equations. For the Fisher-KPP case, Sherratt and Marchant \cite{ShMa96} numerically 
studied the convergence to traveling fronts of solutions with particular initial data in the case of diffusion 
given by $D(u) = u$. For degenerate diffusions of porous medium type and for their generalizations, respectively, Bir\'o \cite{Bir02} and Medvedev \textit{et al.} \cite{MOH03} employed similar techniques to those introduced by Hosono to show that solutions with compactly supported initial data evolve towards the 
Fisher-KPP degenerate front with minimum speed $c_* > 0$ (that is, towards the sharp-front wave). These 
results were extended by Kamin and Rosenau \cite{KaRo04a} to reaction functions of the form $f(u) = u(1-u^m)$, 
with the same porous medium type of diffusion, and with initial data decaying sufficiently fast. So far, no 
rigorous work on stability of smooth fronts is known.

Our paper differs from the aforementioned works in several ways. On one hand our study focuses on the property
of spectral stability, the first step in a general stability program. Consequently, our analysis makes no use 
of parabolic PDE techniques, pertains to the spectral theory of operators, and could be extrapolated to the 
systems case. On the other hand, we study the stability of the whole family of smooth fronts with speed $c > 
c_*$, unlike previous analyses which are restricted to the sharp wave case with $c = c_*$. In addition, we 
consider the generic class of degenerate diffusion coefficients satisfying \eqref{hypD} and introduced by 
S\'{a}nchez-Gardu\~{n}o and Maini. Finally, we conjecture that some of the ideas employed here at the spectral 
level, and designed to deal with the degeneracy of the diffusion, could be applied to more general situations 
(see section \ref{secdiscuss} below for a discussion about this point). This paper is, thus, more related 
in spirit to the work by Meyries \textit{et al.} \cite{MRS14}, by extending their agenda to the 
case of degenerate diffusions, and it is closer in methodology to the program initiated in \cite{AGJ90,KaPro13,San02}.

\subsection*{Plan of the paper} In section \ref{secstructure} we briefly review the existence theory of 
traveling fronts due to S\'anchez-Gardu\~no and Maini \cite{SaMa94a,SaMa94b}. We focus on the main 
structural properties of the waves, such as monotonicity, the lower bound of the speed, as well as its 
asymptotic behavior. In section \ref{secproblem} we pose the stability problem and introduce the partition of 
the spectrum suitable for our needs. Section \ref{secenergy} contains the energy estimates which allow us to 
control the point spectrum. Section \ref{secparab} contains the 
definition of the (parabolic) regularized operator and the proof of generalized convergence when the 
regularization parameter tends to zero. In addition, we compute the Fredholm boundaries for the 
regularized operators and link them to the location of the subset of the 
compression spectrum for the degenerate operator. Section \ref{secweighted} contains the proof of 
Theorem \ref{mainthm}, by choosing appropriate exponentially weighted spaces in which spectral stability 
does hold. In the last section \ref{secdiscuss}, we make some final remarks. Appendix \ref{secsoldecay} 
contains a detailed analysis of the decay of $L^2$-eigenfunctions, which is needed to justify the energy estimates of section \ref{secenergy}.

\section{Structure of Fisher-KPP diffusion-degenerate fronts}
\label{secstructure}

In this section we recall the traveling wave existence theory due to S\'{a}nchez-Gardu\~{n}o and Maini 
\cite{SaMa94a,SaMa95}. The latter is based on the analysis of the local and global phase portraits of the 
associated ODE system; another approach, which 
applies the Conley index to prove the existence of the waves, can be found 
in \cite{ElATa10a,SaMaKa96a}. (For existence results for a more general 
class of equations with advection terms, see Gilding and Kersner \cite{GiKe05}.) 

Let us suppose that $u(x,t) = \varphi(x-ct)$ is a traveling wave solution to \eqref{degRD} with speed $c \in 
\R$. Upon substitution, we find that the profile function $\varphi : \R \to \R$ must be a solution to the 
equation
\begin{equation}
\label{fronteq}
(D(\varphi) \varphi_\xi)_\xi + c \varphi_\xi + f(\varphi) = 0,
\end{equation}
where $\xi = x - ct$ denotes the translation (Galilean) variable. Let us denote the asymptotic limits of the 
traveling wave as
\[
u_\pm := \varphi(\pm \infty) = \lim_{\xi \to \pm \infty} \varphi(\xi).
\]
It is assumed that $u_+$ and $u_-$ are equilibrium points of the reaction function under consideration. 
Written as a first order system, equation \eqref{fronteq} is recast as
\begin{equation}
\label{origsys}
\begin{aligned}
\frac{d \varphi}{d\xi} &= v \\
D(\varphi) \frac{dv}{d\xi} &= -cv -D'(\varphi)v^2 - f(\varphi).
\end{aligned}
\end{equation}
Notice that, due to the degenerate diffusion, this system is degenerate at $\varphi = 0$. Aronson 
\cite{Aron80} overcomes the singularity by introducing the parameter $\tau = \tau(\xi)$, such that
\[
\frac{d\tau}{d \xi} = \frac{1}{D(\varphi(\xi))},
\]
and, therefore, the system is transformed into
\begin{equation}\label{odeSys}
\begin{aligned}
\frac{d\varphi}{d\tau} &= D(\varphi)v \\
 \frac{dv}{d\tau} &= -cv -D'(\varphi)v^2 - f(\varphi).
\end{aligned}
\end{equation}

Heteroclinic trajectories of both systems are equivalent, so the analysis focuses on the study of the 
topological properties of equilibria for system \eqref{odeSys}, which depend upon the reaction function $f$. 
The existence of 
monotone traveling wave solutions is summarized in the following

\begin{proposition}[monotone degenerate Fisher-KPP fronts \cite{SaMa94a}]\label{propSaMa}
If the function $D = D(u)$ satisfies \eqref{hypD} and $f = f(u)$ is of Fisher-KPP type satisfying 
\eqref{Fisherf}, then there exists a unique speed value $c_* > 0$ such that equation \eqref{degRD} has a
monotone decreasing traveling front for each $c > c_*$, with
\[
u_+ = \varphi(+\infty) = 0, \qquad u_- = \varphi(-\infty) = 1,
\]
and $\varphi_\xi < 0$ for all $\xi \in \R$. Each front is diffusion degenerate at $u_+ = 0$, as $\xi \to 
+\infty$.
\end{proposition}

Notice that this is an infinite family of fronts parametrized by the speed $c > c_*$ connecting the 
equilibrium points $u_+ =0$ and $u_-=1$. The fronts are diffusion degenerate in the sense that the diffusion 
coefficient vanishes at one of the equilibrium points, in this case, at $u_+= 0$. 

\begin{remark}
Theorem 2 in \cite{SaMa94a} guarantees the absence of traveling wave solutions when $0 < c < c_*$, as well as 
the 
existence of traveling waves of ``sharp'' type when $c = c_*$. The latter are not considered in the present 
analysis. See \cite{SaMa95,SaMaKa96b} for further information.
\end{remark}

As a by-product of the analysis in \cite{SaMa94a}, one can explicitly determine the asymptotic behavior of 
the traveling fronts as $\xi \to \pm \infty$. This information will be useful later on. 

\begin{lemma}[asymptotic decay]
\label{lemdecayFd}
Let $\varphi = \varphi(\xi)$ a monotone decreasing Fisher-KPP degenerate front, traveling with speed $c > c_* 
> 0$, and with $u_+ =0$, $u_- = 1$. Then $\varphi$ behaves asymptotically as
\[
|\partial_\xi^j(\varphi - u_+)| = |\partial_\xi^j \varphi | = O(e^{-f'(0)\xi/c}), \quad \text{as } \; \xi \to 
+\infty, \; j=0,1,
\]
on the degenerate side; and as,
\[
|\partial_\xi^j(\varphi - u_-)| = |\partial_\xi^j(\varphi - 1)| = O(e^{\eta \xi}), \quad \text{as } \; \xi \to 
-\infty, \; j=0,1,
\]
on the non-degenerate side, with $\eta = (2D(1))^{-1} (-c + \sqrt{c^2 - 4D(1)f'(1)}) > 0$.
\end{lemma}
\begin{proof}
To verify the exponential decay on the the non-degenerate side as $\xi \to -\infty$, note that the end point 
$u_- = 1$ is not diffusion-degenerate and, therefore, the
equilibrium point $P_1 = (1,0)$ for system \eqref{origsys} is hyperbolic. Indeed, writing \eqref{origsys} as 
$\partial_\xi (\varphi, v) = F(\varphi,v)$, then the linearization around $P_1$ is given by
\[
DF_{|(1,0)} = \begin{pmatrix}
0 & 1 \\ -f'(1)/D(1) & -c/D(1)
\end{pmatrix},
\]
with positive eigenvalue $\eta = (2D(1))^{-1} (-c + \sqrt{c^2 - 4D(1)f'(1)}) > 0$. The exponential decay as 
$\xi \to -\infty$ follows by standard ODE estimates around a hyperbolic rest point.

To verify the exponential decay on the degenerate side, notice that $P_0 = (0,0)$ is a non-hyperbolic point 
for system \eqref{odeSys} for all admissible values of the speed $c > c_*$, and we need higher order terms to 
approximate the trajectory along a center manifold. Let us denote the latter as $v = h(\varphi)$; after an 
application of the center manifold theorem, we find that $P_0$ is locally a saddle-node, and the center 
manifold has the form
\[
h(\varphi) = - \frac{f'(0)}{c} \varphi - \frac{1}{2c^3}\big( f''(0) c^2 + 4 D'(0) f'(0)^2 \big) \varphi^2 + 
O(\varphi^3),
\]
for $\varphi \sim 0$ (see \cite{SaMa94a} for details). The trajectory leaves the saddle-node along the center 
manifold for $\varphi \sim 0$. Therefore, for $\xi \to +\infty$, the trajectory behaves as
\[
\varphi_\xi = h(\varphi) \approx - \frac{f'(0)}{c} \varphi \leq 0,
\]
yielding
\[
\varphi = O(e^{-f'(0)\xi/c}), \quad \text{as } \; \xi \to +\infty.
\]
This proves the result.
\end{proof}

\begin{remark}
\label{remH2}
We finish this section by observing that, due to their asymptotic decay, the monotone fronts satisfy 
$\varphi_\xi \in L^2(\R)$ ($\varphi_\xi \to 0$ as $\xi \to \pm \infty$ fast enough). Upon substitution into 
system \eqref{origsys} and by a bootstrapping argument, it can be verified that $\varphi_\xi \in H^2(\R)$ for 
all monotone fronts under consideration. We omit the details.
\end{remark}

\section{The stability problem}
\label{secproblem}

\subsection{Perturbation equations}

Suppose that $v(x,t) = \varphi(x-ct)$ is a monotone traveling front solution to the diffusion degenerate 
Fisher-KPP equation \eqref{degRD}, traveling with speed $c > c_* > 0$. With a slight abuse of notation we 
make the change of variables 
$x \to x-ct$, where now $x$ denotes the Galilean variable of translation. We shall keep this notation for 
the rest of the paper. In the new coordinates, $v$ is a solution to the equation,
\begin{equation}
\label{newRD}
v_t = ( D(v) v_x)_x + cv_x + f(v),
\end{equation}
for which traveling fronts are stationary solutions, $v(x,t) = \varphi(x)$, satisfying,
\begin{equation}\label{profileq}
(D(\varphi)\varphi_x)_x+c\varphi_x+f(\varphi)=0.
\end{equation}
The front connects asymptotic equilibrium points of the reaction: $\varphi(x) \to u_{\pm}$ as $x 
\to \pm \infty$, where $u_+ := 0$, $u_- := 1$, and the front is monotone decreasing $\varphi_x < 0$. 

Let us consider solutions to \eqref{newRD} of the form $\varphi(x)+u(x,t)$, where now $u$ denotes a 
perturbation. Substituting we obtain the nonlinear perturbation equation
\[  u_t = (D(\varphi+u)(\varphi+u)_x)_x + cu_x+ c\varphi_x +f(u+\varphi). \]
Linearizing around the front and using the profile equation \eqref{profileq} we get
\begin{equation}\label{lineareq}
u_t = (D(\varphi)u)_{xx} +cu_x+f'(\varphi)u.
\end{equation}
The right hand side of equation \eqref{lineareq}, regarded as a linear operator acting on an appropriate 
Banach space $X$, naturally defines the spectral problem 
\begin{equation}\label{spectralp}
\lambda u = \mathcal{L}u,
\end{equation}
where $\lambda \in \mathbb{C}$ is the spectral parameter, and
\begin{equation}\label{opL}
\begin{aligned}
\mathcal{L}&: \mathcal{D}(\mathcal{L}) \subset X \to X, \\
\mathcal{L}u &=(D(\varphi)u)_{xx} +cu_x+f'(\varphi)u,
\end{aligned}
\end{equation}
is the linearized operator around the wave, acting on $X$ with domain $\cD(\cL) \subset X$.

Formally, a necessary condition for the front to be stable is that there are no solutions $u \in X$ to 
equation 
\eqref{spectralp} for $\Re \lambda > 0$, precluding the existence of solutions to 
the linear equation \eqref{lineareq} of the form $e^{\lambda t}u$ that grow exponentially in time. This 
condition is known as \textit{spectral stability}, which we define rigorously below. 

\subsection{Resolvent and spectra}

We shall define a particular partition of spectrum suitable for our needs. Let $X$ and $Y$ be Banach spaces, 
and let $\ccC(X,Y)$ and $\ccB(X,Y)$ denote the sets of all closed and bounded linear operators from $X$ to 
$Y$, respectively. For any $\cL \in \ccC(X,Y)$ we denote its domain as $\cD(\cL) \subseteq X$ and its range 
as $\cR(\cL) := \cL (\cD(\cL)) \subseteq Y$. We say $\cL$ is densely defined if $\overline{\cD(\cL)} = X$.

\begin{definition}
 \label{defspec}
Let $\cL \in \ccC(X,Y)$ be a closed, densely defined operator. Its \textit{resolvent} 
$\rho(\cL)$ is defined as the set of all complex numbers $\lambda \in \C$ such that $\cL - \lambda$ is 
injective, $\cR(\cL - \lambda) = Y$, and $(\cL - \lambda)^{-1}$ is bounded. Its \textit{spectrum} is defined 
as $\sigma(\cL) := \C \backslash \rho(\cL)$. Furthermore, we also define the following subsets of the complex 
plane:
\[
 \begin{aligned}
 \ptsp(\cL) := &\{ \lambda \in \C \, : \, \cL - \lambda \; \text{is not injective}\},\\ 
 \spd(\cL) := &\{ \lambda \in \C \, : \, \cL - \lambda \; \text{ is injective, } \cR(\cL - \lambda) \text{ is 
closed, and } \cR(\cL- \lambda) \neq Y\},\\
\sppi(\cL) := &\{ \lambda \in \C \, : \, \cL - \lambda \; \text{ is injective, and } \cR(\cL - \lambda) 
\text{ is not closed}\}.
 \end{aligned}
\]
The set $\ptsp(\cL)$ is called the \textit{point spectrum} and its elements, \textit{eigenvalues}. Clearly, 
$\lambda \in \ptsp(\cL)$ if and only if there exists $u \in \cD(\cL)$, $u \neq 0$, such 
that $\cL u = \lambda u$.
\end{definition}

\begin{remark}
 First, notice that the sets $\ptsp(\cL)$, $\sppi(\cL)$ and $\spd(\cL)$ are clearly disjoint and, since $\cL$ 
is closed, that
\[
 \sigma(\cL) = \ptsp(\cL) \cup \sppi(\cL) \cup \spd(\cL).
\]
Indeed, in the general case it could happen that, for a certain $\lambda \in \C$, $\cL - \lambda$ is 
invertible but $(\cL - \lambda)^{-1}$ is not bounded. But this pathology never occurs if the operator is 
closed: for any $\cL \in \ccC(X,Y)$ with $\cR(\cL) = Y$, if $\cL$ is invertible then $\cL^{-1} \in \ccB(Y,X)$ 
(cf. \cite{Kat80}, pg. 167).
\end{remark}

\begin{remark}
\label{remappcom}
There are different partitions of the spectrum for an unbounded operator besides the classical definition of 
continuous, residual and point spectra (cf. \cite{EE87}). For instance, the set $\sppi(\cL)$ 
is contained in the \textit{approximate spectrum}, defined as
\[
\begin{aligned}
 \sppi(\cL) \subset {\sigma_\mathrm{\tiny{app}}}(\cL) := &\{ \lambda \in \C \, : \, \text{for each $n \in 
\N$ there exists $u_n \in \cD(\cL)$ with $\|u_n\| = 1$} \\ & \; \text{such that $(\cL - \lambda)u_n \to 0$ in 
$Y$ as $n \to +\infty$}\}.
\end{aligned}
\]
The inclusion follows from the fact that, for any $\lambda \in \sppi(\cL)$, the range of $\cL - \lambda$ is 
not closed and, therefore, there exists a \textit{singular sequence}, $u_n \in \cD(\cL)$, $\|u_n\| =1$ such 
that $(\cL - \lambda)u_n \to 0$, which contains no convergent subsequence (see Theorems 5.10 and 
5.11 in Kato \cite{Kat80}, pg. 233). We shall make use of this property of the set $\sppi$ later on (see the proof of Lemma \ref{lemelem} below). It is 
clear that $\ptsp(\cL) \subset {\sigma_\mathrm{\tiny{app}}}(\cL)$, as 
well.

The set $\spd(\cL)$ is clearly contained in what is often called the \textit{compression spectrum} 
\cite{ReRo04} (or \textit{surjective spectrum} \cite{VMue03}):
\[
 \spd(\cL) \subset {\sigma_\mathrm{\tiny{com}}}(\cL) := \{ \lambda \in \C \, : \, \cL - \lambda \; \text{ is 
injective, and } \overline{\cR(\cL-\lambda)} \neq Y\}.
\]
(Since $\overline{\cR(\cL-\lambda)} \neq Y$ it is said that the range has been compressed.)
\end{remark}

Our partition of spectrum splits the classical residual and continuous spectra into two disjoint components, 
making a distinction between points for which the range of $\cL - \lambda$ is closed and those for which it is 
not. These properties 
make this partition suitable for analyzing the stability of the spectrum of the diffusion-degenerate operator 
\eqref{opL}, as we shall see.

\begin{definition}
\label{defspecstab}
 We say that the traveling front $\varphi$ is $X$-\textit{spectrally stable} if 
 \[ 
 \sigma(\mathcal{L}) \subset \{ \lambda \in \C \, : \, \Re{\lambda} \leq 0\}.
  \]
\end{definition}

In this paper we shall consider $X = L^{2}(\mathbb{R};\mathbb{C})$ and $\mathcal{D}(\mathcal{L})= 
H^{2}(\mathbb{R};\mathbb{C})$, 
so that $\mathcal{L}$ is a closed, densely defined operator acting on $L^{2}$. In this fashion, the stability 
analysis of the operator $\cL$ pertains to \textit{localized} perturbations. In what follows, $\sigma(\cL)$ 
will denote the $L^2$-spectrum of the linearized operator \eqref{opL} with 
domain $\cD = H^2$, except where it is otherwise computed with respect to a space $X$ and explicitly denoted as $\sigma(\cL)_{|X}$.

We recall that $\lambda=0$ always belongs to the point spectrum (translation eigenvalue), inasmuch as
\[ 
\mathcal{L} \varphi_x = \partial_x \big( (D(\varphi)\varphi_x)_x+c\varphi_x+f(\varphi) \big)=0,
\]
in view of the profile equation \eqref{profileq} and the fact that $\varphi_x \in 
H^{2}(\mathbb{R}) = \cD(\cL)$. Thus, $\varphi_x$ is the eigenfunction associated to the eigenvalue 
$\lambda=0$.

Finally, we remind the reader that an operator $\mathcal{L} \in \ccC(X,Y)$ is said to be Fredholm if its 
range $\mathcal{R(L)}$ is closed, and both its nullity, $\nul\mathcal{L} = \dim \ker \mathcal{L}$, and 
its deficiency, $\mathrm{def} \,\mathcal{L} = \mathrm{codim} \, \mathcal{R(L)}$, are finite. $\cL$ is said to 
be semi-Fredholm if $\cR(\cL)$ is closed and at least one of $\nul \cL$ and $\mathrm{def} \, \cL$ is finite. 
In both cases the index of $\cL$ is defined as
$\ind \mathcal{L} = \nul \mathcal{L} - \mathrm{def} \, \mathcal{L}$ (cf. \cite{Kat80}).

\begin{remark}
For nonlinear wave stability purposes (cf. \cite{KaPro13}), the spectrum is often partitioned into essential, 
$\ess$, and point spectrum, 
${\widetilde{\sigma}_\mathrm{\tiny{pt}}}$, being the former the set of 
complex numbers $\lambda$ for which $\cL - \lambda$ is either not 
Fredholm or has index different 
from zero, whereas ${\widetilde{\sigma}_\mathrm{\tiny{pt}}}$ is defined as the set of complex numbers for 
which 
$\cL - \lambda$ is Fredholm with index zero and has a 
non-trivial kernel. Note that ${\widetilde{\sigma}_\mathrm{\tiny{pt}}} \subset \ptsp$. This 
definition is due to Weyl \cite{We10}, making $\ess$ a large set but easy to compute, whereas 
${\widetilde{\sigma}_\mathrm{\tiny{pt}}}$ is a discrete set of isolated eigenvalues (see Remark 2.2.4 in 
\cite{KaPro13}). In the present context with 
degenerate diffusion, however, this partition is not particularly useful due to the loss of hyperbolicity 
of the asymptotic coefficients of the operator.
\end{remark}

\section{Energy estimates and stability of the point spectrum}
\label{secenergy}

In this section we show that monotone diffusion-degenerate Fisher-KPP fronts are point spectrally stable. For that purpose we employ energy estimates and an appropriate change of variables. The monotonicity of the front plays a key role.

\subsection{The basic energy estimate}

Let $\varphi= \varphi(x)$ be a diffusion-degenerate monotone Fisher-KPP front and let $\mathcal{L}$ be the corresponding 
linearized operator around $\varphi$, acting on $L^2$. 
Take any fixed $\lambda \in \C$, and assume that there exists a solution $u\in 
\cD(\cL) = H^{2}(\R;\C)$ satisfying the spectral equation $(\cL - \lambda) u = 0$, which we write down as
\begin{equation}
\label{spectralp2}
(\mathcal{L} - \lambda)u= D(\varphi)u_{xx}+ (2 D(\varphi)_x + c) u_x + (D(\varphi)_{xx}+f'(\varphi) - 
\lambda)u = 0.\\
\end{equation}
Consider the change of variables
\begin{equation}
\label{changev}
u(x)= w(x) e^{\theta(x)},
\end{equation}
where $\theta =\theta(x)$ satisfies
\[ 
\theta_x= - \, \frac{c}{2D(\varphi)}, \qquad \text{for all } x\in \R.
\]
Upon substitution we arrive at the equation
\begin{equation}\label{eqw}
D(\varphi)w_{xx}+2 D(\varphi)_x w_x+ H(x)w - \lambda w = 0,
\end{equation}
where
\begin{equation*}
\begin{aligned}
H(x) &= D(\varphi) \theta_x^2 + D(\varphi) \theta_{xx} + (2 D(\varphi)_x + c) \theta_x + D(\varphi)_{xx} + 
f'(\varphi)\\
&= -\frac{c}{2} \frac{D(\varphi)_x}{D(\varphi)}-\frac{c^2}{4 D(\varphi)}+ D(\varphi)_{xx}+f'(\varphi).
\end{aligned}
\end{equation*}
Clearly, the function $\theta$ has the form
\begin{equation}\label{gamma}
\theta(x) = -\frac{c}{2} \int_{x_0}^{x}\frac{ds}{D(\varphi(s))} \, ,
\end{equation}
for any fixed $x_0 \in \R$, and it is well defined for all $x\in \R$. Note that $e^{-\theta(x)}$ may diverge 
as 
$x\to + \infty$. It is true, however, that $w \in H^2$ whenever $u \in H^2$.

\begin{lemma}
\label{lemgoodw}
If $u \in H^2(\R;\C)$ is a solution to the spectral equation $\cL u = \lambda u$, for some fixed $\lambda \in 
\C$, then 
\[
w(x) = \exp \left( \frac{c}{2} \int_{x_0}^x \frac{ds}{D(\varphi(s))}  \right) u(x),
\]
belongs to $H^2(\R;\C)$. Here $x_0 \in \R$ is fixed but arbitrary.
\end{lemma}
\begin{proof}
See Appendix \ref{secsoldecay}.
\end{proof}

In view that $\varphi_x \in \ker \cL$, substitute $\lambda = 0$ into \eqref{eqw} to obtain
\begin{equation}
\label{eqpsi}
D(\varphi)\psi_{xx}+ 2D(\varphi)_x \psi_x+H(x)\psi =0,
\end{equation}
where $\psi := e^{-\theta}\varphi_x \in H^{2}$ (by Lemma \ref{lemgoodw}). Multiply equations \eqref{eqw} and 
\eqref{eqpsi} by 
$D(\varphi)$ and rearrange the terms appropriately; the result is
\begin{equation}
 \label{eqab}
 \begin{aligned}
  (D(\varphi)^{2}w_x)_x +D(\varphi)H(x)w - \lambda D(\varphi)w &= 0,\\
  (D(\varphi)^{2}\psi_x)_x +D(\varphi)H(x)\psi &= 0.
 \end{aligned}
\end{equation}

Since the front is monotone, $\varphi_x < 0$, we have that $\psi \neq 0$ for all $x \in \R$. Therefore, we 
substitute
\[ 
D(\varphi)H(x) = -\frac{(D(\varphi)^{2}\psi_x)_x}{\psi}
\]
into the first equation in \eqref{eqab} to obtain
\begin{equation}
\label{sltype}
(D(\varphi)^{2}w_x)_x -\frac{(D(\varphi)^{2}\psi_x)_x}{\psi}w - \lambda D(\varphi)w = 0.
\end{equation}
Take the complex $L^{2}$-product of $w$ with last equation, and integrate by parts. This yields, 
\begin{equation*}
\begin{aligned}
\lambda \int\limits_{\R} D(\varphi)|w|^{2} dx  &= \int\limits_{\R} 
\overline{w} 
(D(\varphi)^{2}w_x)_x dx - \int\limits_{\R} \psi^{-1} (D(\varphi)^{2}\psi_x)_x |w|^{2} dx\\
&= -\int\limits_{\R} D(\varphi)^{2}|w_x|^{2}dx + \int\limits_{\R} D(\varphi)^{2} \psi_x 
\left(\frac{|w|^{2}}{\psi} \right)_x dx \\
&= \int\limits_{\R} D(\varphi)^{2} \left( \psi_x \left(\frac{|w|^{2}}{\psi} \right)_x-|w_x|^{2} \right) dx,
\end{aligned}
\end{equation*}
where $\overline{w}$ denotes complex conjugate. Using the identity
\begin{equation*}
\psi^{2}\left|\left(\frac{w}{\psi} \right)_x \right|^{2} = - \left( \psi_x \left(\frac{|w|^{2}}{\psi} 
\right)_x-|w_x|^{2} \right),
\end{equation*}
and substituting, we obtain the estimate
\begin{equation}
\label{relaux}
\lambda \int\limits_{\R} D(\varphi)|w|^{2} dx  = -  
\int\limits_{\R} 
D(\varphi)^{2}\psi^{2}\left|\left(\frac{w}{\psi} \right)_x \right|^{2} dx.
\end{equation}

Let us denote, according to custom, the standard $L^2$-product as
\[
\langle u,v \rangle_{L^2} = \int_\R \overline{u} v \, dx, \qquad \|u\|_{L^2}^2 = \langle u,u \rangle_{L^2}.
\]
Hence, we can write relation \eqref{relaux} as
\[
\lambda \langle D(\varphi)w, w \rangle_{L^2} = - \| D(\varphi) \psi (w/\psi)_x \|_{L^2}^2.
\]
Notice that, thanks to Lemma \ref{lemgoodw}, $w \in H^2$, $\psi \in H^2$, and by monotonicity, $\psi < 0$, so 
that the $L^2$-products of last equation are well-defined.

We summarize these calculations into the following
\begin{proposition}[basic energy estimate]
For any $\lambda \in \C$, suppose that there exists a solution $u \in H^2(\R;\C)$ to 
the spectral equation $(\cL - \lambda) u = 0$. Then there holds the energy estimate
\begin{equation}
 \label{basicee}
\lambda \< D(\varphi)w, w\>_{L^2} = -  \| D(\varphi) \psi (w/\psi)_x \|_{L^2}^2,
\end{equation}
where $w = e^{-\theta} u \in H^2(\R;\C)$, $\psi = e^{-\theta} \varphi_x \in H^2(\R)$ is a non-vanishing real 
function, and $\theta = \theta(x)$ is defined in \eqref{gamma}.
\end{proposition}
\begin{remark}
It is to be observed that the monotonicity of the front is crucial to obtain the resolvent type equation 
\eqref{sltype}, and consequently, the basic estimate \eqref{basicee}. In addition, the transformation 
\eqref{changev} is designed to eliminate the transport term $cu_x$ in \eqref{spectralp2}, while keeping the 
term $2D(\varphi)_x u_x$. A transformation that eliminates all the first order terms is not useful to close 
the energy estimate, as the dedicated reader may verify.
\end{remark}

\subsection{Point spectral stability}

The first application of the energy estimate is the following
\begin{theorem}
\label{thmptstab}
All monotone fronts of diffusion degenerate Fisher-KPP equations are point spectrally stable. More precisely,
\begin{equation}\label{ptstab}
\ptsp(\mathcal{L}) \subset ( -\infty,0]
\end{equation}
that is, the $L^2$-point spectrum is real and non-positive.
\end{theorem}
\begin{proof}
Let $\varphi= \varphi(x)$ be a degenerate Fisher-KPP monotone front, and let $\lambda \in 
\ptsp(\cL)$. Therefore, there exists $u \in H^2(\R;\C)$ such that $(\cL - \lambda) u = 0$. Since $D(\varphi) 
\geq 0$ for all $x\in \R$, the energy estimate \eqref{basicee} immediately shows that $\lambda$ is real and non-positive.
\end{proof}

\begin{corollary}
$\lambda=0$ has geometric multiplicity equal to one, that is, $\ker \mathcal{L} = \mathrm{span} \{ 
\varphi_x\}$.
\end{corollary}
\begin{proof}
If we suppose that $\lambda=0$, then estimate \eqref{basicee} yields
\[  
\left ( \frac{w}{\psi}\right )_x =0, \quad \text{a.e. in } \R,
\]
that is, $w = \beta \psi$  for some scalar $\beta$, which implies, in turn, that $u = \beta \varphi_x$. This 
shows that $\lambda =0$ has geometric 
multiplicity equal to one.
\end{proof}

\section{Parabolic regularization and location of $\sigma_\delta$}
\label{secparab}

In this section we introduce a regularization technique that allows us to locate the subset of the compression 
spectrum, $\sigma_\delta$, of the linearized operator around the wave. The method relies on the convergence 
in the generalized sense of the regularized operators.

\subsection{The regularized operator}
\label{secregop}

Let $\varphi = \varphi(x)$ be a diffusion-degenerate monotone Fisher-KPP front. Then, for any $\epsilon > 0$, 
we introduce the following regularization of the diffusion coefficient,
\begin{equation}
\label{regD}
D^\epsilon(\varphi) := D(\varphi) + \epsilon.
\end{equation}
Note that $D^\epsilon(\varphi) > 0$ for all $x \in \R$. Likewise, we also define the following regularized 
operator
\begin{equation}
\label{perturbOpL}
\begin{aligned}
\cL^\epsilon \, &: \, \cD = H^2(\R;\C) \subset L^2(\R;\C) \, \to \, L^2(\R;\C),\\
\cL^\epsilon u &:= (D^\epsilon(\varphi) u)_{xx} + cu_x + f'(\varphi)u.
\end{aligned}
\end{equation}

Notice that, for every $\epsilon > 0$, $- \, \cL^\epsilon$ is a linear, closed, densely defined and 
\textit{strongly elliptic} operator acting on $L^2$. Hence, multiplication by $D^\epsilon (\varphi)^{-1}$ is 
an isomorphism and the Fredholm properties of $\mathcal{L}^\epsilon - \lambda$ and those of the operator 
$\mathcal{J}^\epsilon(\lambda) : \cD \to L^2(\R;\C)$, defined as
\begin{equation}
\label{opJl}
\begin{aligned}
\mathcal{J}^\epsilon(\lambda)u &:=  D^\epsilon (\varphi)^{-1} (\mathcal{L}^\epsilon - \lambda)u \\&=  u_{xx} 
+ 
 D^\epsilon (\varphi)^{-1}a_1(x)u_x +  D^\epsilon (\varphi)^{-1}(a_0(x)- \lambda) u,
\end{aligned}
\end{equation}  
for all $u \in H^2 \subset L^2$, are the same. Here the coefficients $a_0$ and $a_1$ are given by
\[ 
a_1(x) =2 D^\epsilon(\varphi)_x + c,\qquad 
a_0(x) = D^\epsilon(\varphi)_{xx}+f'(\varphi).
 \]

Following Alexander, Gardner and Jones \cite{AGJ90}, it is now customary to recast the spectral problem 
\eqref{opJl} as a first order system of the form
\begin{equation}
\label{firstorder}
W_x = \A^\epsilon(x,\lambda) W,
\end{equation}
where
\[
\A^\epsilon(x,\lambda) = \begin{pmatrix}
0 & 1 \\ D^{\epsilon}(\varphi)^{-1}(\lambda - a_0(x)) & - D^{\epsilon}(\varphi)^{-1}a_1(x)
\end{pmatrix}, \qquad W = \begin{pmatrix} u \\ u_x \end{pmatrix} \in H^1(\R;\C^2).
\]

It is a well-known fact \cite{KaPro13,San02} that the associated first order operators
\[
\mathcal{T}^\epsilon(\lambda) =  \partial_x -\mathbb{A}^\epsilon(\cdot , \lambda), \qquad 
\mathcal{T}^\epsilon(\lambda) : H^1(\R;\C^2) \subset L^2(\R;\C^2) \to L^2(\R;\C),
\]
are endowed with the same Fredholm properties as $\mathcal{J}^\epsilon(\lambda)$ and, consequently, as 
$\cL^\epsilon - \lambda$; see, e.g., Theorem 3.2 in \cite{MRS14}, as well as the references \cite{KaPro13} and 
\cite{San02}. Moreover, these Fredholm properties depend upon the hyperbolicity of the asymptotic matrices 
(cf. \cite{San02}),
\[
\A_\pm^\epsilon(\lambda) = \lim_{x \to \pm \infty} \A^\epsilon(x,\lambda) = \begin{pmatrix}
0 & 1 \\ D^{\epsilon}(u_\pm)^{-1}(\lambda - f'(u_\pm)) & - D^{\epsilon}(u_\pm)^{-1}c
\end{pmatrix}.
\]

For each fixed $\lambda \in \C$, let us denote the characteristic polynomial of $\A_\pm^\epsilon(\lambda)$ as 
$\pi_\pm^\epsilon(\lambda,z) := \det (\A_\pm^\epsilon(\lambda) - zI)$. Then for each $k \in \R$, the 
$\lambda$-roots of
\[
\pi_\pm^\epsilon(\lambda,ik) = -k^2 + ikc D^\epsilon(u_\pm)^{-1} + D^\epsilon(u_\pm)^{-1}(f'(u_\pm)-\lambda) = 
0,
\]
define algebraic curves in the complex plane parametrized by $k \in \R$, more precisely,
\[
\lambda_\pm^\epsilon(k) := - D^\epsilon(u_\pm) k^2 + ick + f'(u_\pm), \quad k \in \R.
\]

Consider the following open connected subset in the complex plane,
\begin{equation}
\label{omegacs}
\Omega := \{ \lambda \in \C \, : \, \Re \lambda > \max \{f'(u_+), f'(u_-)\} \}.
\end{equation}
This is called the region of consistent splitting. Finally, for each fixed $\lambda \in \C$ and $\epsilon > 
0$, let us denote $S_\pm^\epsilon(\lambda)$ and $U_\pm^\epsilon(\lambda)$ as the stable and unstable 
eigenspaces of $\A_\pm^\epsilon(\lambda)$ in $\C^2$, respectively.

\begin{lemma}
\label{lemconsplit}
For each $\lambda \in \Omega$ and all $\epsilon > 0$, the coefficient matrices $\A^\epsilon_\pm(\lambda)$ 
have no center eigenspace and $\dim S_\pm^\epsilon(\lambda) = \dim U_\pm^\epsilon(\lambda) = 1$.
\end{lemma}
\begin{proof}
Notice that 
\[
\max_{k \in \R} \, \Re \lambda_\pm^\epsilon(k) = f'(u_\pm),
\]
independently of $\epsilon > 0$. Therefore, for each $\lambda \in \Omega$ it is clear that 
$\A^\epsilon_\pm(\lambda)$ has no center eigenspace, for all $\epsilon > 0$. By continuity on $\lambda$ and by 
connectedness of $\Omega$, the dimensions of $S_\pm^\epsilon(\lambda)$ and $U_\pm^\epsilon(\lambda)$ remain 
constant in $\Omega$. To compute them, set $\lambda = \eta \in \R$, with $\eta$ sufficiently large. The 
characteristic equation $\pi_\pm^\epsilon(\eta,z) = 0$ has one positive and one negative root:
\[
\begin{aligned}
z_1 &= \tfrac{1}{2}D^\epsilon(u_\pm)^{-1}\big( -c-\sqrt{c^2 + 4 D^\epsilon(u_\pm)(\eta - f'(u_\pm)) }\big) < 
0,\\
z_2 &= \tfrac{1}{2}D^\epsilon(u_\pm)^{-1}\big( -c+\sqrt{c^2 + 4 D^\epsilon(u_\pm)(\eta - f'(u_\pm)) }\big) > 
0,
\end{aligned}
\]
as long as $\eta > 0$ is large, say, $\eta > \max \{|f'(u_\pm)|\}$. The conclusion follows.
\end{proof}

The following lemma characterizes the Fredholm properties of $\cL^\epsilon - \lambda$ for $\lambda$ in the 
region of consistent splitting.

\begin{lemma}
\label{lemLep}
For all $\epsilon > 0$ and for each $\lambda \in \Omega$, the operator $\cL^\epsilon - \lambda$ is Fredholm 
with index zero.
\end{lemma}
\begin{proof}
Let $\lambda \in \Omega$. Since the matrices $\A_\pm^\epsilon(\lambda)$ are hyperbolic, by standard 
exponential dichotomies theory \cite{Cop78} (see also Theorem 3.3 in \cite{San02}), system \eqref{firstorder} 
is endowed with exponential dichotomies on both rays $[0,+\infty)$ and $(-\infty,0]$, with Morse indices 
$i_+(\lambda) = \dim U_+^\epsilon(\lambda) = 1$ and $i_-(\lambda) = \dim S_-^\epsilon(\lambda) = 1$, 
respectively.  Therefore, by Theorem 3.2 in \cite{San02}, we conclude that the operators 
$\mathcal{T}^\epsilon(\lambda)$ are Fredholm with index
\[
\ind \mathcal{T}^\epsilon(\lambda) = i_+(\lambda) - i_-(\lambda) = 0,
\]
showing that $\mathcal{J}^\epsilon(\lambda)$ and $\cL^\epsilon - \lambda$ are Fredholm with index zero, as 
claimed.
\end{proof}

\subsection{Generalized convergence}

We are going to profit from the independence of the Fredholm properties of $\cL^\epsilon - \lambda$ with 
respect to $\epsilon > 0$ in order to conclude some useful information about the Fredholm properties of $\cL - 
\lambda$. First, we recall the succeeding standard definitions (cf. Kato \cite{Kat80}): Let $Z$ be a Banach 
space, and let $M$ and $N$ be any two nontrivial closed subspaces of $Z$. Let $S_M$ be the unitary sphere in 
$M$. Then we define 
\begin{equation*}
\delta(M,N) =\sup_{u\in S_M} \mathrm{dist}(u,N),
\end{equation*}
\begin{equation*}
\hat{\delta}(M,N) = \max \{\delta(M,N),\delta(N,M)\}.
\end{equation*}
$\hat{\delta}$ is called the \textit{gap} between $M$ and $N$. (Here the function $\mathrm{dist}(u, M)$ is the 
usual distance function from $u$ to any closed manifold $M$.) Since in general $\hat \delta(\cdot, \cdot)$  
does not satisfy the triangle inequality, one defines
\begin{equation*}
\mathrm{d}(M,N) =\sup_{u\in S_M} \mathrm{dist}(u,S_N),
\end{equation*}
\begin{equation*}
\mathrm{\hat{d}}(M,N) = \max \{\mathrm{d}(M,N),\mathrm{d}(N,M)\}.
\end{equation*}
$\mathrm{\hat{d}}(M,N)$ is called the \textit{distance} between $M$ and $N$, and satisfies the triangle 
inequality. Furthermore, there hold the inequalities \cite{Kat80},
\begin{equation}
\label{dineq}
\begin{aligned}
\delta(M,N) &\leq \mathrm{d}(M,N) \leq 2 \delta(M,N),\\
\hat{\delta}(M,N) &\leq \mathrm{\hat{d}}(M,N) \leq 2 \hat{\delta}(M,N),
\end{aligned}
\end{equation}
for any closed manifolds $M$ and $N$.

\begin{definition}
Let $X,Y$ be Banach spaces. If $\mathcal{T}, \mathcal{S} \in \mathscr{C}(X,Y)$, then the graphs 
$G(\mathcal{T}), G(\mathcal{S})$ 
are closed subspaces of $X\times Y$, and we set 
\begin{equation*}
\mathrm{d}(\mathcal{T},\mathcal{S}) = \mathrm{d}(G(\mathcal{T}), G(\mathcal{S})) ,
\end{equation*}
\begin{equation*}
\mathrm{\hat{d}}(\mathcal{T},\mathcal{S}) = \max 
\{\mathrm{d}(\mathcal{T},\mathcal{S}),\mathrm{d}(\mathcal{S},\mathcal{T})\}.
\end{equation*}
It is said that a sequence $\mathcal{T}_n \in \mathscr{C}(X,Y)$ \textit{converges in generalized sense} to 
$\mathcal{T}  \in \mathscr{C}(X,Y)$ provided that 
$\mathrm{\hat{d}}(\mathcal{T}_n,\mathcal{T}) \to 0$ as $n \to +\infty$.
\end{definition}
\begin{remark}
It follows from inequalities \eqref{dineq} that $\mathrm{\hat{d}}(\mathcal{T}_n,\mathcal{T}) \to 0$ is 
equivalent to ${\hat{\delta}}(\mathcal{T}_n,\mathcal{T}) \to 0$
\end{remark}

\begin{lemma}
\label{lemconv}
For each fixed $\lambda \in \C$, the operators $\cL^\epsilon - \lambda$ converge in generalized sense to $\cL 
- \lambda$ as $\epsilon \to 0^+$.
\end{lemma}
\begin{proof}
From the definition of $\mathrm{d}(\cdot,\cdot)$ we have
\begin{equation*}
\mathrm{d}(\mathcal{L}^\epsilon -\lambda,\mathcal{L} -\lambda) = \mathrm{d}(G(\mathcal{L}^\epsilon 
-\lambda),G(\mathcal{L} -\lambda)) = \sup_{v\in S_{G(\mathcal{L}^\epsilon -\lambda)}} 
\left ( \inf_{w\in S_{G(\mathcal{L} -\lambda)}} \| v-w \|  \right ).
\end{equation*}

Let $v\in S_{G(\mathcal{L}^\epsilon -\lambda)}$ be such that $v = \{p,(\mathcal{L}^\epsilon -\lambda)p \}$ 
for 
$p\in \mathcal{D}(\mathcal{L}^\epsilon -\lambda) = H^2$, and 
\begin{equation*}
\| v\| ^2_{L^2 \times L^2} = \| p \| ^2_{L^2} + \| (\mathcal{L}^\epsilon -\lambda)p \| ^2_{L^2} =1.
\end{equation*}
Likewise, let $w\in S_{G(\mathcal{L} -\lambda)}$ be such that $w = \{u,(\mathcal{L} -\lambda)u \}$, for $u\in 
\mathcal{D}(\mathcal{L} -\lambda) = H^2$ and
\begin{equation*}
\| w\| ^2_{L^2 \times L^2} = \| u \| ^2_{L^2} + \| (\mathcal{L} -\lambda)u \| ^2_{L^2} =1.
\end{equation*}

Now, we find a upper bound for $\| v-w\|_{L^2 \times L^2}$. Consider, 
\begin{equation}
\label{graphNorm}
\| v-w \| ^2_{L^2 \times L^2} = \| p-u \| ^2_{L^2} +\| (\mathcal{L}^\epsilon -\lambda)p -(\mathcal{L} 
-\lambda)u \| ^2_{L^2}.
\end{equation}
If we keep $v \in S_{G(\mathcal{L}^\epsilon -\lambda)}$ fixed, then it suffices to take $w = \{p,(\mathcal{L} 
-\lambda)p \}$, inasmuch as  $p \in H^2 = \mathcal{D}(\mathcal{L}) = \mathcal{D}(\mathcal{L}^\epsilon)$. Note 
that $(\mathcal{L} -\lambda)p = (\mathcal{L}^\epsilon -\lambda)p - \epsilon p_{xx}$. Therefore expression 
\eqref{graphNorm} gets simplified:
\begin{equation*}
\| v-w \| ^2_{L^2 \times L^2} = \| (\mathcal{L}^\epsilon -\lambda)p -(\mathcal{L} -\lambda)p \| ^2_{L^2} = \| 
\epsilon p_{xx} \| ^2_{L^2}.
\end{equation*}

If we regard $\partial_x^2$ as a closed, densely defined operator on $L^2(\R;\C)$, with domain $\cD = 
H^2(\R;\C)$, then it follows from Remark 1.5 in \cite[p. 191]{Kat80}, that $\partial_x^2$ is 
$(\mathcal{L}^\epsilon -\lambda)-$bounded, i.e., there exist a constant $C>0$ such that 
\begin{equation*}
\| p_{xx} \|_{L^2} \leq C (\| p \| _{L^2} + \| (\mathcal{L}^\epsilon -\lambda)p \| _{L^2} ),
\end{equation*}
for all $p \in H^2$. Consequently
\[ \| p_{xx} \|^2_{L^2} \leq \bar{C} (\| p \| ^2_{L^2}+\| (\mathcal{L}^\epsilon -\lambda)p \| 
^2_{L^2})=\bar{C},
\]
for some other $\bar{C} > 0$ and for $v = (p, (\cL^\epsilon - \lambda)p) \in S_{G(\mathcal{L}^\epsilon 
-\lambda)}$. This estimate implies, in turn, that
\[ 
\| v-w \| ^2_{L^2 \times L^2} = \epsilon^2 \| p_{xx} \| ^2_{L^2} \leq \bar{C}\epsilon^2.
\]
This yields,
\begin{equation*}
\mathrm{d}(\mathcal{L}^\epsilon -\lambda,\mathcal{L} -\lambda) \leq \bar{C}\epsilon^2.
\end{equation*}
In a similar fashion it can be proved that 
$\mathrm{d}(\mathcal{L} -\lambda,\mathcal{L}^\epsilon -\lambda) \leq C\epsilon^2.$
This shows that $\hat{\mathrm{d}}(\mathcal{L}^\epsilon -\lambda,\mathcal{L} -\lambda) \to 0$
as $\epsilon\to 0 ^{+}$, and the conclusion follows.
\end{proof}

For the reader's convenience we state a result from functional analysis (cf. \cite{Kat80}, pg. 235), which 
will be helpful to relate the Fredholm properties of $\cL^\epsilon - \lambda$ to those of $\cL - \lambda$.  
First, we remind the reader the definition of the reduced minimum modulus of a closed operator.
\begin{definition}
For any $\mathcal{S}\in \mathscr{C}(X,Y)$ we define $\gamma= \gamma(\mathcal{S})$, as the greatest number 
$\gamma \in \mathbb{R}$ such that  
\[ 
\| \mathcal{S}u \|  \geq \gamma \: \mathrm{dist}(u, \ker \mathcal{S}), 
\]
for all $u \in  \mathcal{D}(\mathcal{\mathcal{S}})$.
\end{definition}

\begin{theorem}[Kato \cite{Kat80}]
\label{thmindexKato}
Let $\mathcal{T}, \mathcal{S} \in \mathscr{C}(X,Y)$ and let $\mathcal{T}$ be Fredholm (semi-Fredholm). If 
\[
\hat{\delta}(\mathcal{S},\mathcal{T})< \gamma (1+ \gamma^2)^{-1/2},
\] 
where $\gamma=\gamma(\mathcal{T})$, then $\mathcal{S}$ is Fredholm (semi-Fredholm) and $\nul \mathcal{S} 
\leq \nul \mathcal{T}$, $\mathrm{def}\, \mathcal{S} \leq \mathrm{def}\, \mathcal{T}$. Furthermore, there 
exists 
$\delta> 0$ such that $\hat{\delta}(\mathcal{S},\mathcal{T})<\delta$ implies 
\[ 
\ind \mathcal{S} = \ind \mathcal{T}. 
\]
If $X,Y$ are Hilbert spaces then we can take $\delta = \gamma (1+ \gamma^2)^{-\frac{1}{2}}$.
\end{theorem}

\subsection{Location of $\sigma_\delta$}

After these preparations, we are ready to state and prove the main result 
of this section.
\begin{theorem}
\label{alaizq}
Let $\varphi  = \varphi(x)$ be a diffusion-degenerate Fisher-KPP front, and let $\cL \in \ccC(L^2)$ 
be the linearized operator around $\varphi$ defined in \eqref{opL}. Then,
\[
\spd(\cL) \subset \C \backslash \Omega,
\]
where $\Omega$ is defined in \eqref{omegacs}.
\end{theorem}
\begin{proof}
 First observe that
 \[
  \spd(\cL) \subset \{ \lambda \in \C \, : \, \cL-\lambda \, \text{is semi-Fredholm with } \ind (\cL - 
\lambda) \neq 0\}.
 \]
Indeed, by definition if $\lambda \in \spd(\cL)$ then $\cL - \lambda$ is injective, $\cR(\cL-\lambda)$ is 
closed and $\cR(\cL-\lambda) \subsetneqq Y=L^2$. Hence, $\nul(\cL - \lambda) = 0$ and $\cL - \lambda$ is 
semi-Fredholm. Moreover, since $\mathrm{def} \, (\cL - \lambda) = \mathrm{codim} \, \cR(\cL-\lambda) > 0$, we 
have 
that $\ind(\cL-\lambda) \neq 0$.

Now, let us suppose that $\lambda \in \spd(\cL) \cap \Omega$. Since $\cL-\lambda$ is semi-Fredholm, 
$\cR(\cL-\lambda)$ is closed, and this implies that
\[
 \gamma:= \gamma(\cL-\lambda) > 0,
\]
(see Theorem 5.2 in \cite{Kat80}, pg. 231). By Lemma \ref{lemconv}, $\hat{\delta}(\cL^\epsilon - \lambda, \cL 
- \lambda) \to 0$ as $\epsilon \to 0^+$, so we can find $\epsilon > 0$ sufficiently small such that
\[
 \hat \delta(\cL^\epsilon - \lambda, \cL 
- \lambda) < \gamma (1+\gamma^2)^{1/2}.
\]
Since $X = L^2(\R;\C)$ is a Hilbert space, Theorem \ref{thmindexKato} implies that
\[
 \ind(\cL - \lambda) = \ind(\cL^\epsilon - \lambda) = 0,
\]
in view of $\lambda \in \Omega$ and Lemma \ref{lemLep}. This is a contradiction with $\ind(\cL - \lambda) 
\neq 0$. We conclude that 
$\spd(\cL) \subset \C \backslash \Omega$, as claimed.
\end{proof}

\begin{remark}
\label{remstatg}
We observe that the location of $\sigma_\delta(\cL)$ depends upon the sign of $f'(u_\pm)$, as 
it lies to the left of the region of consistent splitting $\Omega$. Since $f'(0) > 0$, the subset 
$\sigma_\delta(\cL)$ of the compression spectrum is unstable. It is, however, sensitive to changes at spatial 
infinity, and it is possible to find exponentially weighted spaces where spectral stability does hold.
\end{remark}

\section{Spectral stability in exponentially weighted spaces}
\label{secweighted}

In this section we prove Theorem \ref{mainthm}. The key ingredient is to find a suitable exponentially 
weighted space in which the spectrum of the linearized operator, computed with respect to the new space, is 
stable. This is accomplished provided that certain conditions on the velocity hold.

\subsection{Exponentially weighted spaces}
\label{secexwesp}
It is well-known \cite{KaPro13,San02} that the Fredholm borders of the compression spectrum 
$\sigma_\delta(\cL)$ may move when the eigenvalue problem is recast in an exponentially weighted space. We 
introduce, according to custom, the function spaces
\[
H^m_a(\R;\C) = \{ v \, : \, e^{ax} v(x) \in H^m(\R;\C)\},
\]
for $m \in \Z$, $m\geq 0$, and any $a \in \R$. The latter are Hilbert spaces endowed with the inner product (and norm),
\[
\langle u,v \rangle_{H^m_a} := \langle e^{ax}u, e^{ax}v \rangle_{H^m}, \qquad \| v \|^2_{H^m_a} := \| e^{ax}v 
\|^2_{H^m} = \langle v,v \rangle_{H^m_a}.
\]
As usual, we denote $L_a^2(\R;\C) = H^0_a(\R;\C)$. Clearly, the norms $\| \cdot \|_{L^2_a}$ for different 
values of $a$ are not equivalent. If we consider $\cL$ as an operator acting on $L_a^2$,
\[
\cL : \cD_a := H^2_a(\R;\C) \subset L^2_a(\R;\C) \, \to \, L^2_a(\R;\C),
\]
and compute its spectrum with respect to the new space for an appropriate value of $a$, then it is known 
\cite{San02} that the Fredholm borders move depending on the sign of $a$ (if $a > 0$, for example, then the 
$\| \cdot \|_{L^2_a}$-norm penalizes perturbations at $-\infty$ while it tolerates exponentially growing 
perturbations at $+\infty$ at any rate less than $a > 0$, shifting, in this fashion, the Fredholm borders to 
the left), whereas the point spectrum is unmoved (cf. \cite{KaPro13}). 

Finally, it is well-known that the analysis of the spectrum of the operator $\cL$ on the space $L^2_a$ is 
equivalent to that of a conjugated operator,
\[
\cL_a := e^{ax} \cL e^{-ax} \, : \, \cD = H^2(\R;\C) \subset L^2(\R;\C) \to L^2(\R;\C),
\]
acting on the original unweighted space. If $\varphi$ is a traveling front and $\cL$ is the associated 
linearized operator acting on $L^2$ defined in \eqref{opL}, then after simple calculations we find that, for 
any $a \in \R$, the associated conjugated operator $\cL_a$ is given by
\begin{equation}
\label{conjop}
\begin{aligned}
\cL_a &: \cD(\cL_a) := H^2(\R;\C) \subset L^2(\R;\C) \to L^2(\R:\C),\\
\cL_a u &= D(\varphi)u_{xx} + \Big( 2 D(\varphi)_x - 2a D(\varphi) + c\Big) u_x + \\ & \quad + \Big( a^2 
D(\varphi) - 2a D(\varphi)_x - ac + D(\varphi)_{xx} + f'(\varphi) \Big)u, 
\end{aligned}
\end{equation}
for all $u \in H^2$. $\cL_a$ is clearly a closed, densely defined operator in $L^2$.

\subsection{Calculation of the Fredholm curves}

In order to analyze how the Fredholm borders limiting $\sigma_\delta$ move under the influence of the weight 
function $e^{ax}$, let us consider the regularized conjugated operator $\cL_a^\epsilon$ acting on $L^2$, for 
$0 < \epsilon \ll 1$, small (which results from substituting $D(\varphi)$ by $D^\epsilon(\varphi) = D(\varphi) 
+ \epsilon$ in \eqref{conjop}; see section \ref{secregop}). Thus, if for any $\lambda \in \C$ we define 
$\mathcal{J}^\epsilon_a(\lambda) := D^\epsilon(\varphi)^{-1}(\cL_a^\epsilon - \lambda)$, then its explicit 
expression is
\[
\mathcal{J}^\epsilon_a(\lambda) u = u_{xx} + D^\epsilon(\varphi)^{-1} b^\epsilon_{1,a}(x) u_x + 
D^\epsilon(\varphi)^{-1} (b^\epsilon_{0,a}(x)  - \lambda) u,
\]
with
\[
\begin{aligned}
b^\epsilon_{1,a}(x) &:= 2D^\epsilon(\varphi)_x + c - 2a D^\epsilon(\varphi),\\
b^\epsilon_{0,a}(x) &:= a^2D^\epsilon(\varphi) - 2aD^\epsilon(\varphi)_x -ac + D^\epsilon(\varphi)_{xx} + 
f'(\varphi).
\end{aligned}
\]

Like in section \ref{secregop}, we recast the spectral problem as a first order system of the form
\[
W_x = \A^\epsilon_a(x,\lambda) W,
\]
where
\[
W = \begin{pmatrix} u \\ u_x
\end{pmatrix} \in H^2(\R;\C^2), \quad 
\A^\epsilon_a (x,\lambda) = \begin{pmatrix} 0 & 1 \\ D^\epsilon(\varphi)^{-1} (\lambda - b^\epsilon_{0,a}(x)) 
& -D^\epsilon(\varphi)^{-1} b^\epsilon_{1,a}(x)
\end{pmatrix}.
\]
Since,
\[
\begin{aligned}
\lim_{x \to \pm \infty} b^\epsilon_{0,a}(x) &= a^2 D^\epsilon(u_\pm) - ac + f'(u_\pm) =: 
b^{\epsilon,\pm}_{0,a},\\
\lim_{x \to \pm \infty} b^\epsilon_{1,a}(x) &= c - 2aD^\epsilon(u_\pm) =: b^{\epsilon,\pm}_{1,a},
\end{aligned}
\]
the constant coefficients of the asymptotic systems can be written as
\[
\A^{\epsilon,\pm}_a(\lambda) = \begin{pmatrix}
0 & 1 \\ D^\epsilon(u_\pm)^{-1} (\lambda - b^{\epsilon,\pm}_{0,a}) & - D^\epsilon(u_\pm)^{-1} 
b^{\epsilon,\pm}_{1,a}
\end{pmatrix}.
\]
Let us denote $\pi^{\epsilon,\pm}_a (\lambda,z) = \det (\A^{\epsilon,\pm}_a(\lambda) - zI)$, so that, for each 
$k \in \R$
\[
\pi^{\epsilon,\pm}_a (\lambda,ik) = -k^2 + ik(c D^\epsilon(u_\pm)^{-1} - 2a) + a^2 - D^\epsilon(u_\pm)^{-1} 
(\lambda + ac - f'(u_\pm)).
\]
Thus, the Fredholm borders, defined as the $\lambda$-roots of $\pi^{\epsilon,\pm}_a(\lambda, ik) = 0$, are 
given by
\begin{equation*}
\lambda^{\epsilon,\pm}_a(k) := D^\epsilon(u_\pm) (a^2 - k^2) -ac + f'(u_\pm) + ik(c - 2a D^\epsilon(u_\pm)), 
\quad k \in \R.
\end{equation*}
Notice that
\[
\max_{k \in \R} \, \Re \lambda^{\epsilon,\pm}_a(k) = D^\epsilon(u_\pm) a^2 -ac + f'(u_\pm);
\]
therefore we denote the region of consistent splitting for each $a \in \R$ and $\epsilon \geq 0$ as
\[
\Omega(a,\epsilon) = \left\{ \lambda \in \C \, : \, \Re \lambda > \max \, \{ D^\epsilon(u_+) a^2 -ac + 
f'(u_+), \, D^\epsilon(u_-) a^2 -ac + f'(u_-)\} \right\}.
\]

By a similar argument as in section \ref{secparab} (see Lemma \ref{lemconsplit}), for each $\lambda \in 
\Omega(a,\epsilon)$ the coefficient matrices $\A^{\epsilon\pm}_a(\lambda)$ have no center eigenspace and $\dim 
S^{\epsilon,\pm}_a(\lambda) = \dim U^{\epsilon,\pm}_a(\lambda) =1$, where $S^{\epsilon,\pm}_a(\lambda)$ and 
$U^{\epsilon,\pm}_a(\lambda)$ denote the stable and unstable eigenspaces of $\A^{\epsilon\pm}_a(\lambda)$, 
respectively.

Moreover, by taking the limit as $\epsilon \to 0^+$, we claim that
\begin{equation}
\label{contention}
\sigma_\delta(\cL_a) \subset \C \backslash \Omega(a),
\end{equation}
where $\Omega(a) := \Omega(a,0)$. Indeed, using the same arguments as in the proof of Lemma \ref{lemLep}, we 
conclude that, for $\lambda \in \Omega(a,\epsilon)$, $\cL_a^\epsilon - \lambda$ is Fredholm in $L^2$ with 
index zero. By keeping $a \in \R$ 
constant, one can verify that the operators $\cL^\epsilon_a - \lambda$ converge in generalized sense to 
$\cL_a 
- \lambda$ as $\epsilon \to 0^+$ (see Lemma \ref{lemconv}; we omit the details as the proof is not only 
analogous, but the same). Furthermore, by the same arguments as in the proof of Theorem \ref{alaizq}, for $0 
< 
\epsilon \ll 1$ sufficiently small the Fredholm properties of $\cL^\epsilon_a - \lambda$ and $\cL_a - 
\lambda$ 
are the same. 

Suppose that $\lambda \in \sigma_\delta(\cL_a) \cap \Omega(a,\epsilon)$, with $0 < \epsilon \ll 1$, small. 
Then, by following the proof of Theorem \ref{alaizq}, we find that $\mathrm{ind} (\cL_a - \lambda) = 
\ind(\cL_a^\epsilon -\lambda) = 0$, and at the same time, that $\mathrm{ind}(\cL_a - \lambda) \neq 0$ because 
$\sigma_\delta(\cL_a)$ is contained in the set of complex $\lambda$ for which $\cL_a - \lambda$ is 
semi-Fredholm with non-zero index. This is a contradiction, which yields $\sigma_\delta(\cL_a) \subset \C 
\backslash \Omega(a,\epsilon)$ for all $0 < \epsilon \ll 1$ sufficiently small. By continuity, taking the 
limit as $\epsilon \to 0^+$ we obtain $\sigma_\delta(\cL_a) \subset \C \backslash \Omega(a)$, as claimed.

Henceforth, it suffices to choose $a \in \R$ appropriately in order to stabilize $\sigma_\delta(\cL_a)$.

\subsection{Choice of the weight $a \geq 0$}
\label{secrproof}

In the Fisher-KPP case, $u_+ = 0$, $u_-=1$, with $f'(0) > 0$ and $f'(1) < 0$. The fronts travel with any speed 
$c > 
c_* > 0$. Under these conditions we have
\[
 \Omega(a) = \{ \lambda \in \C \, : \, \Re \lambda > \max \{ f'(0) - ac, \, D(1) a^2 - ac + f'(1)\} \, \},
\]
for any $a \in \R$. In view of \eqref{contention}, in order to have spectral stability we need to find $a \in \R$ such that $f'(0) - ac < 0$ 
and $p(a) := D(1)a^2 - ac + f'(1) < 0$, simultaneously. Since $p(0) = f'(1) < 0$ we have that $p(a) < 0$ for 
all $a \in [0, a_0)$, where $a_0$ is the first positive root of $p(a) = 0$, that is,
\[
 a_0(c) = (2 D(1))^{-1} \big( c + \sqrt{c^2 - 4D(1)f'(1)} \, \big).
\]
We need to find $a$ such that $0 < f'(0)/c < a < a_0(c)$. Thus, it is necessary and sufficient that 
$a_0(c) > f'(0)/c$. (This imposes a condition on the speed $c$.) It is easy to verify, by elementary computations, that this is true 
provided that the front travels with speed $c$ such that
\begin{equation}
 \label{Fdfast}
 c^2 > \frac{D(1)f'(0)^2}{f'(0)-f'(1)}.
\end{equation}

Therefore, we have shown that if the speed $c > c_*$ satisfies condition \eqref{Fdfast} then we can always choose $a \in \R$ such that
\begin{equation}
\label{conda}
0 < \frac{f'(0)}{c} < a < (2 D(1))^{-1} \big( c + \sqrt{c^2 - 4D(1)f'(1)} \, \big).
\end{equation}

Hence we have proved the following
\begin{lemma}
\label{lemsdel}
For any diffusion-degenerate monotone Fisher-KPP front, traveling with speed $c \in \R$ satisfying
\begin{equation}
\label{condFd2}
c > \max \Big\{ \, c_*, \, \,\frac{f'(0) \sqrt{D(1)}}{\sqrt{f'(0) - f'(1)}} \,\Big\} > 0,
\end{equation}
we can choose an appropriate weight $a \in \R$, satisfying condition \eqref{conda}, such that the linearized operator around the front satisfies
\begin{equation}
\label{stabsigdel}
\sigma_\delta(\cL_a)_{|L^2} = \sigma_\delta(\cL)_{|L^2_a} \subset \{ \lambda \in \C \, : \, \Re \lambda < 0 \}.
\end{equation}
\end{lemma}

Now, since the point spectrum is unmoved under conjugation (see Remark 3.1.17 in \cite{KaPro13}, pg. 54), and in view of Theorem \ref{thmptstab}, we immediately have the following
\begin{proposition}\label{propunmoved}
For the monotone traveling fronts under consideration, we have
\begin{equation}
\label{stabsigpta}
\ptsp(\cL_a)_{|L^2} = \ptsp (\cL)_{|L^2} \subset (-\infty, 0],
\end{equation}
for any $a \in \R$ and where $\cL$ is the linearized operator around the front acting on $L^2$.
\end{proposition}

\begin{remark}
This invariance of the point spectrum follows from the fact that, if $u \in \cD(\cL) = H^2(\R)$ is an eigenfunction of $\cL$ associated to an eigenvalue $\lambda \in \ptsp(\cL)$, then $v = e^{ax}u$ is an eigenfunction of $\cL_a$ with same eigenvalue, as $\cL_a v = e^{ax} \cL e^{-ax} v = e^{ax} \cL u = \lambda v$. As Kapitula and Promislow \cite{KaPro13} point out, $v$ is an eigenfunction unless the essential spectrum has moved to include $\lambda$, situation in which $v$ generically does not decay at $x=+\infty$ or at $x = -\infty$. In our case, the choice of $a \in \R$ satisfying \eqref{conda} (in particular, $a > 0$) clearly implies that $v = e^{ax}u \in H^2(-\infty,x_0)$ for any $x_0 > 0$, whenever $u \in H^2$. On the degenerate side, as $x \to +\infty$, we have precise information on the decaying behavior of eigenfunctions provided by Lemma \ref{lemdecayu} in the Appendix; such eigenfunctions can be written in the form \eqref{formuFd}, yielding $|u| \leq C \zeta$ with $\zeta \in H^2(x_0,+\infty)$.  For any $a \in \R$ satisfying \eqref{conda} (in particular, $a > f'(0)/c$) we obtain $\zeta e^{ax} \in H^2(x_0,+\infty)$, thanks to the fast decay of $\zeta$ (see Lemma \ref{lemdecayu} in the Appendix), as the dedicated reader may verify. Therefore, $v = e^{ax} u \in H^2(\R)$. 
\end{remark}

\subsection{Stability of $\sigma_\pi(\cL_a)$ and proof of Theorem \ref{mainthm}}

We are left to prove that, for our choice of exponentially weighted $L^2$-space, $\sigma_\pi(\cL_a)$ is stable as well. For that purpose let us write the operator $\cL_a : L^2 \to L^2$ as
\[
\cL_a = b_2(x) \partial_x^2 + b_1(x) \partial_x + b_0(x), \qquad \cL_a : \cD(\cL_a) = H^2(\R) \subset L^2(\R) \to L^2(\R),
\]
where,
\[
\begin{aligned}
b_2(x) &= D(\varphi),\\b_1(x) &= 2D(\varphi)_x - 2aD(\varphi) + c,\\
b_0(x) &= a^2 D(\varphi) - 2a D(\varphi)_x - ac + D(\varphi)_{xx} + f'(\varphi).
\end{aligned}
\]
Let us denote,
\[
b_j^\pm := \lim_{x \to \pm \infty} b_j(x), \qquad j = 0,1,2,
\]
yielding
\[
\begin{aligned}
b_2^+ &= 0, &  b_2^- &= D(1),\\
b_1^+ &= c, &  b_1^- &= c - 2 a D(1),\\
b_0^+ &= f'(0) - ac, &  b_0^- &= a^2 D(1) - ac + f'(1).
\end{aligned}
\]
By exponential decay of the profile we have also that
\begin{equation}
\label{bexpdec}
|b_j(x) - b_j^\pm| \leq C e^{-  \mu |x|}, \qquad \text{as } \, x \to \pm \infty,
\end{equation}
for some uniform $C,  \mu > 0$, and all $j = 0,1,2$.  Whence, let us define the piecewise constant coefficients
\begin{equation}
\label{pcwsecoef}
b_j^\infty (x) := \begin{cases}
b_j^-, & x < 0,\\ b_j^+, & x > 0,
\end{cases}
\qquad j=0,1,2,
\end{equation}
and the operator $\cL_a^\infty : \cD(\cL_a^\infty) = H^2(\R) \subset L^2(\R) \to L^2(\R)$, as
\[
\cL_a^\infty := b_2^\infty(x) \partial_x^2 + b_1^\infty(x) \partial_x + b_0^\infty(x),
\]
regarded as a closed, densely defined operator in $L^2(\R)$. Associated with it, we define the polynomials
\[
P_\pm(\xi) := \sum_{j=0}^2 b_j^\pm \xi^j, \qquad \xi \in \R,
\]
which determine the operator $\cL_a^\infty$ in the following sense. If $\mathscr{F}$ denotes  the Fourier transform, then for all $u \in C_0^\infty(\R)$,
\begin{equation}
\label{unitary}
\cL_a^\infty u = \mathscr{F}^{-1} M(P_\pm) \mathscr{F} u,
\end{equation}
where $M(P_\pm)$ denotes the maximal operator of multiplication by $P_\pm(\xi)$ (see Definition III.9.1 in \cite{EE87}, pg. 134), having domain
\[
\cD(M(P_\pm)) := \left\{ u \in L^2(\R) \, : \, \int_{-\infty}^0 |P_-(\xi) u(\xi)|^2 \, d\xi + \int_0^{+\infty} |P_+(\xi)u(\xi)|^2 \, d\xi \, < \, +\infty \right\},
\]
which is dense in $L^2(\R)$ (see Theorem III.9.2 in \cite{EE87}).

\begin{remark}
For any $u \in C_0^\infty(\R)$, at any point of continuity of $b_j^\infty(x) \partial_x^j u$, we have the representation in terms of the inverse Fourier transform given by
\[
b_j^\infty(x) \partial_x^j u (x)= \int_0^{+\infty} e^{2\pi i \xi x} (i\xi)^j b_j^+ (\mathscr{F}u)(\xi) \, d \xi +  \int_{-\infty}^0 e^{2\pi i \xi x} (i\xi)^j b_j^- (\mathscr{F}u)(\xi) \, d \xi,
\]
for all $j = 0,1,2$, except at $x = 0$, where the last integral converges to $\tfrac{1}{2}( \partial_x^j u(0^-) + \partial_x^j u(0^+) )$ (the average of the limits as $x \to 0^{\pm}$). Thus,
 \[
\mathscr{F}^{-1} M(P_\pm) \mathscr{F} u= \begin{cases}
{\displaystyle \sum_{j=0}^2 b_j^- \partial_x^j u}, & x < 0,\\ & \\ {\displaystyle \sum_{j=0}^2 b_j^+ \partial_x^j u}, & x > 0,
\end{cases}
\]
and, consequently, $\cL_a^\infty u$ and $\mathscr{F}^{-1} M(P_\pm) \mathscr{F} u$ agree a.e. in $x \in \R$, for all $u \in C_0^\infty(\R)$. Notice that, since $\mathscr{F}$ is a unitary isomorphism from $L^2(\R)$ to $L^2(\R)$, $\cL_a^\infty$ and $M(P_\pm)$ are unitarily equivalent.
\end{remark}

Let us denote now $S_\pm$ as the closure of curves in the complex plane given by
\[
S_\pm = \overline{\Big\{ \sum_{j=0}^2 b_j^\pm(i\xi)^j \, : \,  \xi \in \R \Big\}} \subset \C.
\]
$S_\pm$ consist of a finite number of algebraic curves in $\C$, more precisely,
\begin{equation}
\label{algSpm}
\begin{aligned}
S_+ &=  \overline{\Big\{ ic\xi + f'(0) - ac  \, : \, \xi \in \R \Big\}},\\
S_- &=  \overline{\Big\{ -D(1)\xi^2 + i(c - 2 aD(1)) \xi + a^2D(1) - ac + f'(1)  \, : \, \xi \in \R \Big\}}.
\end{aligned}
\end{equation}

It is a well-known fact that, in the case of differential operators with constant coefficients, their spectrum is completely determined by the dispersion curves $S_\pm$ (see, e.g., Theorem 1 in \cite{BaGa64}, or Theorem IX.6.2 in \cite{EE87}). For our case, this is exactly the content of the following proposition (we sketch its proof for completeness, following \cite{EE87} closely).

\begin{proposition}
\label{propspeccurve}
$\sigma(\cL_a^\infty) = S_+ \cup S_-$.
\end{proposition}
\begin{proof}
Let $\lambda \notin S_\pm:= S_+ \cup S_-$. Then there exists $\delta > 0$ such that $|\lambda - P_\pm(\xi)| > \delta$ for all $\xi \in \R$. Hence, for all $u \in \cD(M(P_\pm))$ (dense in $L^2$), we have
\[
\begin{aligned}
\| (M(P_\pm) - \lambda)u \|_{L^2}^2 &= \int_{-\infty}^0 |P_-(\xi) - \lambda|^2 |u(\xi)|^2 \, d\xi + \int_0^{+\infty} |P_+(\xi) - \lambda|^2 |u(\xi)|^2 \, d\xi \\ &\geq \delta^2 \|u \|_{L^2}^2.
\end{aligned}
\]

For any $\lambda \in \C$, the set $\{ \xi \in \R \, : \, P_\pm(\xi) = \lambda \}$ has Lebesgue measure equal to zero. Applying Theorem III.9.2 in \cite{EE87}, we have that $\cR( M(P_\pm) - \lambda)$ is dense in $L^2(\R)$. For any $\lambda \notin S_\pm$, $(M(P_\pm) - \lambda)^{-1}$ is closed and bounded on $\cR( M(P_\pm) - \lambda)$. Thus,  it must have a closed domain. Therefore, $\lambda \in \rho(M(P_\pm))$ or $\sigma(M(P_\pm)) \subset S_\pm$. Now take $\widetilde{\lambda} \in S_+ \cup S_-$. Then there exists $\xi_0 \in \R$ such that $\widetilde{\lambda} = P_+(\xi_0)$ or $\widetilde{\lambda} = P_-(\xi_0)$. Suppose, without loss of generality, that $\widetilde{\lambda} = P_+(\xi_0)$. Then for any given $\epsilon > 0$ there exists $\delta > 0$ such that $|P_+(\xi) - \widetilde{\lambda}| < \epsilon$ if $|\xi - \xi_0| < \delta$. Hence, any $u \in \cD(M(P_\pm))$ with support in $(\xi_0 - \delta, \xi_0 + \delta)$ satisfies
\[
\| (M(P_\pm) - \widetilde{\lambda}) u \|_{L^2}^2 = \int_{\xi_0 - \delta}^{\xi_0 + \delta} |P_+(\xi) - \widetilde{\lambda}|^2 |u(\xi)|^2 \, d\xi < \epsilon^2 \| u \|_{L^2}^2,
\] 
showing, in particular, that $\widetilde{\lambda} \in \sigma(M(P_\pm))$. Since the spectrum is a closed set, we conclude that $\sigma(M(P_\pm)) = S_+ \cup S_-$. Finally, as $\cL_a^\infty$ and $M(P_\pm)$ are unitarily equivalent, they have the same spectrum, and we obtain the result.
\end{proof}

The following lemma relates the subset of the approximate spectrum, $\sigma_\pi(\cL_a)$, to the spectrum of the asymptotic operator $\cL_a^\infty$.

\begin{lemma}
\label{lemelem}
$\sigma_\pi(\cL_a) \subset \sigma(\cL_a^\infty)$.
\end{lemma}
\begin{proof}
Let $\lambda \in \sigma_\pi(\cL_a)$. Then, by definition, $\cR(\cL_a - \lambda)$ is not closed and we deduce the existence of a singular sequence, $u_n \in \cD(\cL_a) = H^2(\R)$, with $\| u_n \|_{L^2} = 1$ for all $n \in \N$, such that $(\cL_a - \lambda) u_n \to 0$ in $L^2$ as $n \to +\infty$, containing no convergent subsequence (see Remark \ref{remappcom}). In view that they are localized functions and that the energy $L^2$-norm is invariant under translations, $u_n(\cdot) \to u_n(\cdot + y)$, we may assume that 
\[
\|u_n \|_{L^2(|x| \geq n+1)} > 0, \qquad \forall \, n \in \N.
\]
(Indeed, since $\| u_n \|_{L^2} = 1$, there exists an open set $O = (x_0 - R(n), x_0 + R(n))$ such that $\| u_n \|_{L^2(O)} > 0$. Redefining $u_n$ by, say, the translation $u_n(\cdot) \to u_n(\cdot + n+1 - x_0 + R(n))$ we get $\|u_n\|_{L^2(n+1,+\infty)} > 0$.)

Now, since $\| u_n \|_{H^2} \geq \|u_n \|_{L^2} = 1$, let us define
\[
v_n := \frac{u_n}{\| u_n \|_{H^2}} \in \cD(\cL_a) = H^2(\R),
\]
so that $\| v_n \|_{H^2} = 1$ for all $n \in \N$ (uniformly bounded $H^2$-norm). Let $\chi_n = \chi_n(x)$ be such that $\chi_n \in C^\infty(\R)$ for all $n \in \N$,  $\chi_n(x) = 0$ for $x \in [-n,n]$, $\chi_n(x) = 1$ for $|x| \geq n+1$, and $0 \leq \chi_n \leq 1$. Notice that 
\[
\begin{aligned}
\| \chi_n v_n \|_{L^2}^2 &= \| v_n \|_{L^2(|x| \geq n+1)}^2 + \| \chi_n v_n \|^2_{L^2(n \leq |x| \leq n+1)} \\ &= \|u_n \|_{H^2}^{-2} \| u_n \|_{L^2(|x| \geq n+1)}^2 + \| \chi_n v_n \|^2_{L^2(n \leq |x| \leq n+1)} \\&> 0,
\end{aligned}
\]
so that we can define
\[
w_n(x) := \frac{\chi_n(x) v_n(x)}{\| \chi_n v_n \|_{L^2}}, \qquad x \in \R,
\]
and $w_n \in \cD(\cL_a) = H^2(\R)$. We now observe that, since $\chi_n$ is smooth, we may compute
\[
\begin{aligned}
(\cL_a^\infty - \cL_a) (\chi_n v_n) &= \sum_{j=0}^2 (b_j^\infty  - b_j ) \partial_x^j (\chi_n v_n)\\
&= (b_1^\infty  - b_1 ) (\partial_x \chi_n) v_n + (b_2^\infty  - b_2 )( v_n \partial_x^2 \chi_n + 2 \partial_x \chi_n \partial_x v_n) + \\ &\;\;\; + \sum_{j=0}^2 (b_j^\infty  - b_j ) \chi_n \partial_x^j v_n. 
\end{aligned}
\]
Therefore,
\[
\begin{aligned}
\| (\cL_a^\infty - \cL_a)(\chi_n v_n) \|_{L^2}^2 &\leq \sum_{j=0}^2 \| (b_j^\infty - b_j) \chi_n \partial_x^j v_n \|_{L^2}^2 + \\ &  \; \;+ \int\limits_{n \leq |x| \leq n+1} \!\!\! |b_1^\infty(x) - b_1(x)|^2 |\partial_x \chi_n(x)|^2 |v_n(x)|^2 \, dx \; \; + \\ &  \; \;+ \int\limits_{n \leq |x| \leq n+1} \!\!\!|b_2^\infty(x) - b_2(x)|^2 |\partial_x \chi_n(x)|^2 |\partial_x v_n(x)|^2 \, dx \; \; + \\
&  \; \;+ \int\limits_{n \leq |x| \leq n+1}\!\!\! |b_2^\infty(x) - b_2(x)|^2 |\partial_x^2 \chi_n(x)|^2 |v_n(x)|^2 \, dx.
\end{aligned}
\]

By exponential decay \eqref{bexpdec}, for each $j =0,1,2$ we have the bound
\[
\begin{aligned}
\| (b_j^\infty - b_j) \chi_n \partial_x^j v_n \|_{L^2}^2 &\leq \int\limits_{|x| \geq n} |b_j^\infty(x) - b_j(x)|^2 |\partial_x^j v_n(x)|^2 \, dx\\
&\leq C e^{-2\mu n} \|v_n\|_{H^2}^2 = C e^{-2\mu n} \; \longrightarrow 0, \qquad \text{as } \;\; n \to +\infty,
\end{aligned}
\]
inasmuch as $\| v_n \|_{H^2} = 1$ (uniformly bounded $H^2$-norm for the $v_n$'s). The integrals on the set $\{n \leq |x| \leq n+1\}$ are clearly bounded by
\[
\int\limits_{n \leq |x| \leq n+1} \!\!\! |b_j^\infty(x) - b_j(x)|^2 |\partial_x^k \chi_n(x)|^2 ( |v_n|^2 + |\partial_x v_n|^2) \, dx \leq C e^{-2\mu n} \|v_n \|_{H^2}^2 = C e^{-2\mu n} \to 0,
\]
as $n \to +\infty$, for $k = 1,2$. Therefore, we conclude that
\[
\|(\cL_a^\infty - \cL_a) (\chi_n v_n) \|_{L^2} \longrightarrow 0, \qquad \text{as } \; n \to +\infty.
\]

Finally, observe that
\[
\|(\cL_a - \lambda) (\chi_n v_n) \|_{L^2} = \| (\cL_a - \lambda) (\|u_n\|_{H^2}^{-1} \chi_n u_n) \|_{L^2} \leq \| (\cL_a - \lambda) u_n \|_{L^2} \to 0, 
\]
as $n \to +\infty$, because $0 \leq \chi_n \leq 1$, $\| u_n \|_{H^2}^{-1} \leq 1$, and $u_n$ is a singular sequence. This yields,
\[
(\cL_a^\infty - \lambda) w_n = (\cL_a^\infty - \cL_a) w_n + (\cL_a - \lambda) w_n \to 0, \qquad \text{in } \; L^2,
\]
as $n \to +\infty$. 

Therefore, we have a sequence $w_n \in \cD(\cL_a^\infty) = H^2(\R)$ such that $\| w_n \|_{L^2} = 1$ and $(\cL_a^\infty - \lambda)w_n \to 0$ in $L^2$ as $n \to +\infty$. This implies, by definition of the approximate spectrum (see Remark \ref{remappcom}), that $\lambda \in {\sigma_\mathrm{\tiny{app}}}(\cL_a^\infty) \subset \sigma(\cL_a^\infty)$, as claimed.
\end{proof}

\begin{corollary}[stability of $\sigma_\pi(\cL_a)$]
\label{corstabpi}
For any diffusion-degenerate monotone front, traveling with speed satisfying \eqref{condFd2}, and for any weight $a \in \R$ satisfying \eqref{conda}, there holds
\[
\sigma_\pi(\cL_a) \subset \{ \lambda \in \C \, : \, \Re\lambda < 0 \},
\]
where $\cL_a$ is the (conjugated) linearized operator around the front in $L^2$ given by \eqref{conjop}.
\end{corollary}
\begin{proof}
Follows directly from Lemma \ref{lemelem} and Proposition \ref{propspeccurve}, and by observing that, under condition \eqref{condFd2} for the speed, we can always choose a weight  $a \in \R$ satisfying \eqref{conda}, condition which implies, in turn, that the curves \eqref{algSpm} are clearly contained in the stable half plane. Thus, $ \sigma_\pi(\cL_a) \subset \sigma(\cL_a^\infty) = S_+ \cup S_- \subset \{ \lambda \in \C \, : \, \Re\lambda < 0 \}$. 
\end{proof}

\subsection*{Proof of Theorem \ref{mainthm}}

Under conditions \eqref{condFd2} and \eqref{conda} for the traveling speed $c \in \R$ and the exponential weight $a \in \R$, respectively, we may apply Corollary \ref{corstabpi}, Lemma \ref{lemsdel}, Proposition \ref{propunmoved} and Theorem \ref{thmptstab} to conclude that
\[
\sigma(\cL)_{|L^2_a} = \sigma(\cL_a)_{|L^2} = \ptsp(\cL_a)_{|L^2} \cup \sigma_\delta(\cL_a)_{|L^2} \cup \sigma_\pi(\cL_a)_{|L^2}  \subset \{ \lambda \in \C \, : \, \Re\lambda \leq 0 \}.
\]

This finishes the proof of the main Theorem \ref{mainthm}.

\begin{remark}
The theorem guarantees that all traveling fronts with speed satisfying condition \eqref{condFd2} are 
spectrally stable in an appropriate weighted space $L_a^2$. The condition depends on the choice of the 
reaction $f$ and on the density-dependent diffusion $D$ under consideration. For example, in the particular case 
of the diffusion-degenerate Fisher-KPP equation,
\begin{equation}
\label{k1Fisher}
u_t = (\alpha uu_x)_x + u(1-u),
\end{equation}
with $\alpha > 0$, that is, the special case of \eqref{degRD} with $D(u) = \alpha u$ and $f(u) = u(1-u)$, it 
is known (see, e.g., \cite{Aron80,GiKe04,GiKe05}) that the threshold speed is given by $c_* = 
\sqrt{\alpha/2}$. Thus, in this case,
\[
\frac{f'(0) \sqrt{D(1)}}{\sqrt{f'(0) - f'(1)}} = \frac{\sqrt{\alpha}}{\sqrt{2}} = c_*,
\]
and, by Theorem \ref{mainthm}, all traveling waves for equation \eqref{k1Fisher} with speed $c > c_*$ are 
$L_a^2$-spectrally stable for some $a \geq 0$.

There might be choices of $D$ and $f$ for which the maximum in \eqref{condFd2} is not $c_*$ and, therefore, 
the waves traveling with speed $c_* < c < f'(0) \sqrt{D(1)}/\sqrt{f'(0) - f'(1)}$ are spectrally unstable in 
any weighted space. This case is associated to the presence of absolute instabilities and the location of the leftmost limit 
of the Fredholm borders as $a$ varies (see \cite{KaPro13}, section 3.1). For example, 
Sanchez-Gardu\~{n}o and Maini \cite{SaMa94b} calculate, for the choices $D(u) = u + \varepsilon u^2$ and $f(u) 
= u(1-u)$, with $\varepsilon > 0$ small, that the threshold velocity is $c_* = (1 + \varepsilon/5)/\sqrt{2}$. 
Therefore 
\[
\max \Big\{ \, c_*, \, \,\frac{f'(0) \sqrt{D(1)}}{\sqrt{f'(0) - f'(1)}} \,\Big\} = \max \Big\{ \, 
\frac{1}{\sqrt{2}}(1 + \varepsilon/5)\, , \, \frac{\sqrt{1+\varepsilon}}{\sqrt{2}} \,\Big\} = 
\frac{\sqrt{1+\varepsilon}}{\sqrt{2}},
\]
for all small $\varepsilon > 0$. Thus, stability holds for all fronts traveling with speed $c > 
\sqrt{1+\varepsilon} / \sqrt{2}$, whereas those which are slower are spectrally unstable in any weighted 
space. This case covers the normalized Shigesada function \eqref{Dbeta} mentioned in the introduction, with $D(u) = \varepsilon (u^2 + bu)$, and $b = 1/\varepsilon > 0$ large.
\end{remark}

\section{Discussion}
\label{secdiscuss}

In this paper we have shown that the spectrum of the linearized differential operator around any 
diffusion-degenerate Fisher-KPP traveling front with speed satisfying condition \eqref{condFd2} is located in 
the stable complex half plane, $\{ \lambda \in \C \, : \, \Re \lambda \leq 0\}$, when it is computed with 
respect to an appropriate exponentially weighted $L^2$-space. In other words, if the front satisfies 
\eqref{condFd2} then we can always find a local energy space under which spectral stability holds. A few 
remarks, however, are in order. First, it is important to notice that our main result does not imply the 
existence of a ``spectral gap", that is, that $\sigma \subset \{ \Re \lambda \leq - \omega < 0\} \cup \{0\}$, 
for some $\omega > 0$. This is a limitation of the technique used. Moreover, we have not proved that there is 
accumulation of the continuous spectrum near $\Re \lambda = 0$ either. The existence 
of a spectral gap is an important issue to be resolved prior to studying the nonlinear (orbital) stability, 
inasmuch as it is well-known that such spectral gap simplifies the nonlinear study with the use 
of standard exponentially decaying semigroup tools (see, e.g., \cite{He81,KaPro13,PW94,San02} and the 
references therein). The seasoned reader might rightfully 
ask why the parabolic regularization technique (see section \ref{secparab}) is not used to locate the whole 
Weyl's essential spectrum. The answer to that question is precisely that, due to the degeneracy, it might 
happen that the reduced minimum modulus, $\gamma(\cL^\epsilon - \lambda)$, tends to zero as $\epsilon \to 
0^+$, as it does for points in $\sigma_\pi(\cL)$. Therefore, one must sort out the points for which $\cL - 
\lambda$ is a closed range operator. As a consequence, the set which must be controlled with the use of energy 
estimates, namely $\ptsp$, is larger and not necessarily composed of isolated points 
with finite multiplicity only, like in the standard approach for strictly parabolic problems. 

In addition, we conjecture that the ideas introduced here to handle the degeneracy of the 
diffusion can be applied to other circumstances. For instance, it is clear that other reaction functions might 
as well be taken into account, such as the bistable (or Nagumo) type \cite{McKe70,NAY62}; this case will be 
addressed in a companion paper \cite{LeP2}. Another possible application is the case of traveling fronts 
with doubly-degenerate diffusions, which arise naturally in bacterial aggregation models \cite{LMP1}, and 
whose existence has been already studied (cf. \cite{MalMar1,Manso10}). We believe that the analysis presented 
here can be applied to those situations as well, by taking care of the points in the argumentation where 
monotonicity of $D = D(\cdot)$ should be dropped, and by extending the methods to more orders of degeneracy. 
Finally, the stability of traveling fronts for systems with degenerate diffusion tensors is an important open 
problem whose investigation could profit from some of the ideas developed in this paper.

\section*{Acknowledgements}

The authors thank Margaret Beck for useful discussions. The work of J. F. Leyva was partially supported by 
CONACyT (Mexico) through a scholarship for doctoral studies, grant no. 220865. The work of RGP was partially 
supported by 
DGAPA-UNAM, program PAPIIT, grant IN-104814.

\appendix
\section{Proof of Lemma \ref{lemgoodw}}
\label{secsoldecay}

In this section, we verify that, for fixed $\lambda \in \C$, if $u \in H^2$ is a solution to 
the spectral equation $(\cL - \lambda)u = 0$ then
\[
w(x) = \exp \left( \frac{c}{2} \int_{x_0}^x \frac{ds}{D(\varphi(s))}  \right) u(x) = e^{-\theta(x)} u,
\]
belongs to $H^2$ as well. Here $\cL$ is the linearized operator around a traveling front $\varphi$ for the 
Fisher-KPP equation \eqref{degRD}, traveling with speed $c > c_* > 0$, and $x_0 \in \R$ is fixed but 
arbitrary. 

For Fisher-KPP diffusion-degenerate fronts one has $u_+ = 0$, $u_- = 1$, with $f'(0) > 0$, $f'(1) < 0$ and 
they are monotone decreasing with $\varphi_x < 0$. These fronts are diffusion-degenerate as $x \to +\infty$ in 
view that $D(u_+) = D(0) = 0$.

On the non-degenerate side, as $x \to -\infty$, notice that
\[
\int_{x_0}^x D(\varphi(s))^{-1} \, ds \leq 0,
\]
for all $x \leq x_0$. Therefore,
\[
1 \geq \Theta(x) := \exp  \left( \frac{c}{2} \int_{x_0}^x \frac{ds}{D(\varphi(s))}  \right),
\]
This yields $|w(x)| \leq \Theta(x) |u(x)| \leq |u(x)|$ for $x \in (-\infty, x_0)$ and, consequently, $w \in L^2(-\infty,x_0)$. Now, for $x_0 \in \R$ fixed, let
\[
\delta_0 := \inf_{x \in (-\infty, x_0]} D(\varphi(x)) > 0.
\]
Then for all $x < x_0$ we have
\[
\Theta(x) \leq \exp \Big( - \frac{c}{2\delta_0} |x-x_0|\Big) \, \to 0, \qquad \text{as} \; \, x \to - \infty.
\]
Also from monotonicity of $D = D(\cdot)$ we have
\[
0 < \frac{c}{2} \frac{\Theta(x)}{D(1)} \leq \Theta_x(x) = \frac{c}{2} \frac{\Theta(x)}{D(\varphi(x))} \leq \frac{c}{2 \delta_0} \Theta(x) \to 0,
\]
as $x \to -\infty$. We conclude that $w_x = \Theta_x u + \Theta u_x \in L^2(-\infty, x_0)$. Analogously, it can be easily verified that $w_{xx} \in L^2(-\infty, x_0)$ (we omit the details).

On the degenerate side, as $x \to +\infty$, however, observe that $\int_{x_0}^x D(\varphi(s))^{-1} \, ds \geq 0$ for all $x \geq x_0$ and, thus, the analysis of the decay at $+\infty$ of solutions to resolvent type equations is more delicate. The details are provided by the following lemma.

\begin{lemma}
\label{lemdecayu}
Suppose that $\varphi$ is a monotone, Fisher-KPP diffusion-degenerate front. Then, for any fixed $\lambda \in \C$, any $H^2$-solution $u$ to the spectral equation $(\cL - \lambda) u = 0$ can be written as
\begin{equation}
\label{formuFd}
u(x) = C \exp \left( - \frac{c}{2} \int_{x_0}^x \frac{ds}{D(\varphi(s))} \, \right) \zeta(x),
\end{equation}
for $x > x_0$, $x_0 \in \R$ fixed, $x_0 \gg 1$ sufficiently large, with some constant $C \in \C$, and where $\zeta = \zeta(x) \in H^2(x_0,+\infty)$ decays to zero as $x \to +\infty$ like
\begin{equation}
\label{decayofzeta}
\zeta(x) \sim e^{f'(0)x/2c} \exp \left( - \frac{c^2}{2D'(0)f'(0)} e^{f'(0)x/c} \right).
\end{equation}
\end{lemma}

Here we use the standard notation in asymptotic analysis \cite{Erde56}, in which the symbol ``$\sim$" means  ``behaves like" as $x \to x_*$, more precisely, $f \sim g$ as $x \to x_*$ if $f - g = o(|g|)$ as $x \to x_*$ (or equivalently, $f/g \to 1$ as $x \to x_*$ if both functions are positive).

In view of Lemma \ref{lemdecayu}, for $x_0 \gg 1$ fixed we have $w(x) = C\zeta(x) \in H^2(x_0,+\infty)$. Therefore we conclude that $w \in H^2(\R;\C)$ as claimed. This finishes the proof of Lemma \ref{lemgoodw}. We are left to prove Lemma \ref{lemdecayu}.

\begin{proof}[Proof of Lemma \ref{lemdecayu}]
Consider the change of variables $u = \varphi_x v$. Upon substitution into the spectral equation $(\cL - \lambda)u = 0$ and using $\cL \varphi_x = 0$, we obtain 
\[
v_{xx} + \rho(x)v_x - \frac{\lambda}{D(\varphi)}v = 0,
\]
where,
\[
\rho(x) = \frac{2(D(\varphi)\varphi_x)_x}{D(\varphi) \varphi_x} + \frac{c}{D(\varphi)}.
\]
Now, let us define
\[
v(x) = \exp \left( - \frac{1}{2} \int_{x_0}^x \rho(s) \, ds \right) z(x).
\]
Substituting we obtain the second order equation
\begin{equation}
\label{eqzz}
z_{xx} - F(x,\lambda) z = 0,
\end{equation}
with
\[
F(x,\lambda) = \frac{\lambda}{D(\varphi)} + \frac{1}{2} \rho_x + \frac{1}{4} \rho^2.
\]

Now, it is to be observed that
\[
\begin{aligned}
\exp \left( - \frac{1}{2}\int_{x_0}^x \rho(s) \, ds \right) &= \frac{|(D(\varphi)\varphi_x)(x_0)|}{|D(\varphi)\varphi_x|} \exp \left( - \frac{c}{2} \int_{x_0}^x \frac{ds}{D(\varphi(s))} \, ds \right) \\ &= - \frac{C_0}{D(\varphi) \varphi_x} \exp \left( - \frac{c}{2} \int_{x_0}^x \frac{ds}{D(\varphi(s))} \, ds \right),
\end{aligned}
\]
where $C_0 = |(D(\varphi)\varphi_x)(x_0)| > 0$, and thanks to monotonicity of the profile ($\varphi_x < 0$). Let us now define
\[
\tTh(x):= \exp \left( - \frac{c}{2} \int_{x_0}^x \frac{ds}{D(\varphi(s))} \, ds \right).
\]
Observe that $\tTh \to 0$ as $x \to +\infty$, and also that
\[
u = \varphi_x v = \varphi_x \exp \left( - \frac{1}{2}\int_{x_0}^x \rho(s) \, ds \right) z(x) = - C_0 \frac{\tTh(x)}{D(\varphi)} z(x)  =: C_0 \zeta(x) \tTh(x).
\]
Here $\zeta = -z/D(\varphi)$ and $z$ is the decaying solution to equation \eqref{eqzz}. Since by definition $w(x) = \tTh(x)^{-1} u(x)$, we arrive at
\[
w(x) = C_0 \zeta(x).
\]
Thus, the goal is to show that $\zeta$ decays as $x \to +\infty$ fast enough, so that $\zeta \in L^2(x_0, +\infty)$. 

First we observe that, substituting the profile equation \eqref{profileq}, we may write
\[
\rho = \frac{2 (D(\varphi)\varphi_x)_x}{D(\varphi)\varphi_x} + \frac{c}{D(\varphi)} = - \frac{c}{D(\varphi)} - \frac{2f(\varphi)}{D(\varphi)\varphi_x}.
\]
By Lemma \ref{lemdecayFd}, the wave decays as $\varphi = O(e^{-f'(0)x/c}) \to 0$ as $x \to +\infty$. Making Taylor expansions near $\varphi = 0$ of the form 
\begin{equation}
\label{taylorexp}
D(\varphi) = \varphi(D'(0) + O(\varphi)), \qquad D(\varphi)^{-1} = \varphi^{-1}D'(0)^{-1} + O(1),
\end{equation}
we find that
\[
\rho(x) = \frac{c}{D'(0)} \frac{1}{\varphi(x)} + O(1) \to +\infty, \qquad \text{as} \;\, x \to +\infty. 
\]
Moreover,
\[
\frac{1}{4} \rho(x)^2 = \frac{c^2}{4 D'(0)^2} \frac{1}{\varphi(x)^2} \to +\infty, \qquad \frac{\lambda}{D(\varphi)} = \frac{\lambda}{D'(0)} \frac{1}{\varphi(x)} + O(1),
\]
as $x \to +\infty$. Notice that $\lambda/D(\varphi)$ diverges at order $O(1/|\varphi|)$ for $\lambda$ fixed. Upon differentiation of the expression for $\rho$ and substitution of the Taylor expansions around $\varphi$, one can show that
\[
\rho_x = \frac{3f'(0)}{D'(0)} \frac{1}{\varphi} + O(1), \qquad \varphi \sim 0^+.
\]
Thus, we reckon that
\[
F(x,\lambda) = \frac{\lambda}{D(\varphi)} + \frac{1}{2} \rho_x + \frac{1}{4} \rho^2 = \frac{c^2}{4D'(0)^2} \varphi^{-2} + \Big( \frac{\lambda}{D'(0)} - \frac{3f'(0)}{2D'(0)} \Big) \varphi^{-1} + O(1),
\]
for $\varphi \sim 0$ as $x \to +\infty$. The leading term is $\varphi^{-2} c^2 /(4D'(0)^2) > 0$, and since $\varphi = O(e^{-f'(0)x/c})$, then $F(x,\lambda)$ diverges at order $O(e^{2f'(0)x/c})$ as $x \to +\infty$. For fixed $\lambda \in \C$, the leading term does not depend on $\lambda$ and has a definite sign. Therefore, by taking real and imaginary parts, we can assume, without loss of generality, that $z$ decays as the real solution to the equation
\begin{equation}
\label{realzeq}
z_{xx} - F(x)z = 0,
\end{equation}
with 
\[
F(x) := \frac{c^2}{4 D'(0)^2} e^{2 f'(0)x/c}.
\]
We now apply the following theorem by Coppel (cf. \cite{Cop65}, pg. 122):

\begin{theorem}[Coppel \cite{Cop65}]
Let $F(x) > 0$, $F \in C^2$, be such that
\[
\int_{x_0}^x |F^{-3/2}F''| \, dx < +\infty.
\]
Then the homogeneous equation 
\begin{equation}
\label{homeqz}
z_{xx} - F(x)z = 0,
\end{equation}
has a fundamental system of solutions satisfying
\begin{equation}
\label{copdecay}
\begin{aligned}
z &\sim F(x)^{-1/4} \exp \left( \pm \int_{x_0}^x F(s)^{1/2} \, ds \right),\\
z_x &\sim F(x)^{1/4} \exp \left( \pm \int_{x_0}^x F(s)^{1/2} \, ds \right).
\end{aligned}
\end{equation}
\end{theorem}

In order to simplify the notation, let us define
\[
\beta = \frac{2f'(0)}{c} > 0,
\]
so that  $\varphi = O(e^{-\beta x/2})$, as $x \to +\infty$. Since $F > 0$ diverges at order $O(e^{\beta x})$, then one can verify that $F^{-3/2} F'' \sim e^{-\beta x/2}$ is integrable in $(x_0, +\infty)$, if $x_0 \gg 1$ is chosen sufficiently large. Let $z_1(x)$ and $z_2(x)$ be two linearly independent solutions in $[x_0, +\infty)$ to the homogeneous equation \eqref{homeqz}, decaying and diverging at $+\infty$, respectively. Then, by Coppel's theorem, $z_1(x)$ behaves like
\[
\begin{aligned}
z_1(x) &\sim F(x)^{-1/4} \exp \left( - \int_{x_0}^x F(s)^{1/2} \, ds \right)\\
&= \big( \frac{2 D'(0)}{c}\big)^{1/2} e^{-\beta x / 4} \exp \left( - \frac{c}{2 D'(0)} \int_{x_0}^x e^{\beta s/2} \, ds \right)\\
&\leq \big( \frac{2 D'(0)}{c}\big)^{1/2} e^{-\beta x / 4} \exp \left( - \frac{c}{\beta D'(0)} e^{\beta x /2} \right),
\end{aligned}
\]
that is,
\begin{equation}
\label{decayz1}
z_1(x) \sim e^{-\beta x / 4} \exp \left( - \frac{c}{\beta D'(0)} e^{\beta x /2} \right) \to 0,
\end{equation}
as $x \to +\infty$. Likewise, Coppel's theorem implies that
\begin{equation}
\label{decayz1x}
\partial_x z_1(x) \sim e^{\beta x / 4} \exp \left( - \frac{c}{\beta D'(0)} e^{\beta x /2} \right) \to 0,
\end{equation}
as $x \to +\infty$. 

Therefore, the decaying solution to \eqref{realzeq} (and, consequently, the decaying solution to \eqref{eqzz}), behaves like
\[
z(x) \sim e^{-\beta x/4} \exp \left( - \frac{c}{\beta D'(0)} e^{\beta x/2}\right), \qquad x \to +\infty.
\]
This yields,
\[
\begin{aligned}
\zeta(x) = - \frac{z(x)}{D(\varphi(x))} &\sim e^{\beta x/4} \exp \left( - \frac{c}{\beta D'(0)} e^{\beta x/2}\right)\\
&\sim e^{f'(0) x/2c} \exp \left( - \frac{c^2}{2 D'(0) f'(0)} e^{f'(0) x/c}\right),
\end{aligned}
\]
when $x \to +\infty$, as claimed.

It can be shown, using the decay rate \eqref{copdecay} for the derivatives of solutions to the homogeneous equation, that $\zeta_x \in L^2(x_0, +\infty)$ as well, with a decay of the form $e^{kx} \exp ( - \bar{C} e^{f'(0)x/c})$, with $k > 0$, $\bar{C} > 0$. A similar procedure leads to $\zeta_{xx} \in L^2(x_0,+\infty)$. The details are omitted as the proof is analogous (the rapidly decaying term $\exp (- \bar{C} e^{f'(0)x/c})$ controls the possibly blowing up terms of form $e^{kx}$, so that derivatives remain in $L^2$). This concludes the proof.
\end{proof}

\def\cprime{$'$}

%
%
%
%
%
%

\end{document}